\documentclass[11pt]{article} 

\usepackage[utf8]{inputenc} 

\usepackage{geometry} 
\usepackage{graphicx} 

\usepackage{booktabs} 
\usepackage{array} 
\usepackage{paralist} 
\usepackage{verbatim} 
\usepackage{subfig} 
\usepackage{url}
\usepackage{amsmath}
\usepackage{amssymb}

\usepackage{amsthm}

\newtheorem{thm}{Theorem}[section]

\theoremstyle{definition}
\newtheorem{defn}[thm]{Definition}

\title{Sums of Isometric Pairs of Lattices}
\author{Paul Lewis}

\begin{document}
\maketitle
\setcounter{tocdepth}{1}
\tableofcontents

\section{Introduction}
This paper is an investigation of a procedure for constructing lattices by means of taking the sum of a pair of isometric lattices, which may be defined as follows:
\begin{defn}Let $L$ be a lattice, and let $h\in O(L)$ be an isometry of $L$.  We define the sublattices $M$ and $N$ of $L \perp L$ by
\begin{equation*}
M:=\{ (x,x)\mid x\in L\}, N:=\{(x,hx)\mid x\in L\}.
\end{equation*}
We then define $K:=M+N$. \end{defn}
In this paper, parentheses such as $(v,w)$ may refer either to the inner product of vectors or, if $v$ and $w$ are both vectors in $\mathbb{R}^n$, to the vector in $\mathbb{R}^{2n}$ whose first $n$ coordinates are the coordinates of $v$ and whose last $n$ coordinates are those of $w$.  Context will usually dictate the distinction between these interpretations. Where necessary, subscripts (e.g. $K_L$) will be used to clarify which $M$, $N$, or $K$ we are referring to.

We will see that this construction produces a wide variety of interesting lattices.  This process has a number of applications.  For example, the Barnes-Wall family of lattices can be constructed in this way:  the lattice $BW_{2^d}$ can be obtained by letting $L=BW_{2^{d-1}}$ and making an appropriate choice of $h$.

This paper presents various general results pertaining to this construction and discusses several examples of it applied to various well-known lattices of small rank.  In some cases, the lattice $K$ is a known lattice, and in other cases, it is difficult to identify $K$ explicitly.  In many cases, it is possible to embed $K$ in a well-known overlattice, or find a well-known sublattice of small index.

The author acknowledges the support of the National Science Foundation REU program at the University of Michigan and the helpful mentorship of Prof. Robert L. Griess, Jr.

\subsection{Definitions}
The term \emph{lattice} in this paper refers to a finitely generated free abelian group with a rational-valued positive definite symmetric bilinear form.  Any such lattice can be viewed as the integer span of a basis of $\mathbb{R}^n$, in which case the bilinear form is simply the standard inner product.

A lattice is $integral$ if $(v,w)\in\mathbb{Z}$ for all $v,w\in L$.  The \emph{even sublattice} of $L$ is the set of vectors $v\in L$ with even norm ($(v,v)\in 2\mathbb{Z}$).  An \emph{even lattice} is a lattice which is equal to its even sublattice.

Given a lattice $L$, the \emph{dual lattice} $L^*$ is defined as the set of vectors $v$ in the ambient rational vector space $\mathbb{Q}\otimes L$ such that $(v,w)\in\mathbb{Z}$ for all $w\in L$.  The \emph{discriminant group} $\mathcal{D}(L)$ is the group $L^*/L$.

The \emph{Gram matrix} $G_L$ of a lattice $L$ is the symmetric square matrix whose $i,j$-entry is the inner product $(v_i,v_j)$, where $\{v_1,\ldots,v_n\}$ is a basis for $L$.  The determinant of a lattice $L$ is the determinant of the Gram matrix $G_L$.  The \emph{Smith normal form} of a square integer matrix $M$ is the unique matrix $S=PMQ$ formed from $M$ by multiplying on the left and right by invertible (over $\mathbb{Z}$) integer matrices $P,Q$, such that $S$ is a diagonal matrix whose diagonal entries $d_1,\ldots,d_n$ are positive and satisfy $d_1\mid d_2\mid\cdots\mid d_n$.  The \emph{Smith invariant sequence} is this sequence of diagonal entries.\footnote{For more details on this topic, refer to a basic text on the theory of principal ideal domains, such as \cite{HartleyHawkes}.}

If $L$ has Smith invariant sequence $\{d_1,\ldots,d_n\}$, there exists a basis $\{v_1,\ldots,v_n\}$ of $L$ and a basis $\{u_1,\ldots,u_n\}$ of $L^*$ such that $v_1=d_1u_1,\ldots,v_n=d_nu_n$, and hence $\mathcal{D}(L) \cong \mathbb{Z}/d_1\mathbb{Z} \oplus \cdots \oplus \mathbb{Z}/d_n\mathbb{Z}$. Thus $|\mathcal{D}(L)|=d_1\cdots d_n=det(L)$.

A lattice is \emph{rectangular} if it has a basis whose Gram matrix is diagonal.

If we can express $L$ as a sum of nonzero lattices $L_1+L_2$, where $(v,w)=0$ for all $v\in L_1, w \in L_2$, we say $L$ is \emph{orthogonally decomposable} and we write $L=L_1\perp L_2$.

The \emph{isometry group} of $L$, denoted $O(L)$, is the group of endomorphisms of $L$ which preserve the inner product.  We say two lattices $L$ and $M$ are \emph{isometric} and write $L\cong M$ if there exists a bijective linear map $f:L\to M$ which preserves the inner product.

Given an integral lattice $L$, a \emph{root} is a vector in $L$ of norm 2.  Given a root $v\in L$, we can define an isometry $r_v$, where $r_v(x)=x-(x,v)v$.  This is called the \emph{reflection at $v$}.

We frequently refer to the root lattices $A_n$, $D_n$, and $E_n$ in this paper; definitions and basic facts about these lattices can be found in \cite{Griess}.

\subsection{Basic Remarks}
Note that of these lattices $M,N$ is isometric to $\sqrt{2}L$.  If we take $h=1$, $M=N$, so we have $K=M=N\cong \sqrt{2}L$.  Also, for $h=-1$, for all $v=(x,x)\in M, w=(y,-y)\in N$, we have the inner product $(v,w)=(x,y)-(x,y)=0$, so the lattices $M$ and $N$ are orthogonal and we have $K\cong \sqrt{2}L \perp \sqrt{2}L$.  Since all vectors in $M$ and $N$ have even norm, $K$ is an even lattice.

If $h$ and $g$ are conjugate in the isometry group, the corresponding lattices $K_h$ and $K_g$ are isometric.  To see this, suppose $g=k^{-1}hk$.  Since $k^{-1}$ is an isometry of $L$, we have $k^{-1}L=L$, so we can write
\begin{equation*}
M_g=\{ (k^{-1}x,k^{-1}x)\mid x\in L\}, N_g=\{(k^{-1}x,gk^{-1}x)\mid x\in L\}=\{(k^{-1}x,k^{-1}hx)\mid x\in L\}.
\end{equation*}
We can extend the isometry $k^{-1}$ of $L$ to an isometry $\phi := k^{-1} \perp k^{-1}$ of $L \perp L$, and we can see that $M_g=\phi(M_h)$ and $N_g=\phi(N_h)$.  Thus $K_g=\phi(K_h)$, and hence $K_g\cong K_h$.

\subsection{Some Theorems about Lattices}
We now quote several results about lattices which are used throughout this paper.  They are stated without proof for the sake of brevity.  Proofs of all of these results can be found in \cite{Griess}.
\begin{thm}[Index-determinant formula]\label{IndexDet}Let $L$ be a rational lattice, and $M$ a sublattice of $L$ of finite index $|L:M|$.  Then \begin{equation*}det(L)|L:M|^2=det(M).\end{equation*}
\end{thm}
\begin{defn}For integers $n\geq 1$ and real numbers $d\geq 0$, the Hermite function is given by
\begin{equation*}H(n,d):=(\tfrac{4}{3})^{\frac{n-1}{2}}\cdot d^{\frac{1}{n}}.\end{equation*}
For a lattice $L$, we define
\begin{equation*}\mu (L) := \textrm{min}\{(x,x)\mid x\in L, x\neq 0\}. \end{equation*}
\end{defn}
\begin{thm}[Hermite]\label{Hermite}If $L$ is a rational lattice of rank $n$, then $\mu (L) \leq H(n,det(L))$.
\end{thm}
\begin{thm}\label{Unimodular}Let $L$ be an integral lattice of rank no more than $8$ and determinant $1$.  Then $L\cong \mathbb{Z}^n$ or $L\cong E_8$.
\end{thm}
\begin{thm}\label{E7uniqueness}Let $L$ be an integral lattice of rank $7$ and determinant $2$.  Then $L$ is rectangular or $L\cong E_7$.
\end{thm}
\begin{thm}\label{E6uniqueness}Let $L$ be an integral lattice of rank $6$ and determinant $3$.  Then $L$ is rectangular, $L\cong A_2 \perp \mathbb{Z}^4$, or $L\cong E_6$.
\end{thm}
\begin{thm}\label{Rank3}Let $J$ be a rank $3$ integral lattice.  If $det(J)\in \{1,2,3\}$, then $J$ is rectangular or $J$ is isometric to $\mathbb{Z}\perp A_2$.  If $det(J)=4$, $J$ is rectangular or is isometric to $A_3$.
\end{thm}
\begin{thm}\label{Rank4}Suppose that $X$ is an integral lattice which has rank $4$ and determinant $4$.  Then $X$ embeds with index $2$ in $\mathbb{Z}^4$.  If $X$ is odd, $X$ is isometric to one of $2\mathbb{Z} \perp \mathbb{Z}^3, A_1 \perp A_1 \perp \mathbb{Z}^2, A_3 \perp \mathbb{Z}$.  If $X$ is even, $X\cong D_4$.
\end{thm}
\begin{thm}\label{2group}Suppose that $X$ is an integral lattice which has rank $m \geq 1$ and there exists a lattice $W$, so that $X \leq W \leq X^*$ and $W/X \cong 2^r$, for some integer $r \geq 1$. Suppose further that every nontrivial coset of $X$ in $W$ contains a vector with noninteger norm. Then $r = 1$.
\end{thm}

\section{General Results}\label{General}
The following are some general theorems pertaining to the construction defined above.  Our first result gives a complete characterization of the lattices obtained from $\mathbb{Z}^n$ when we choose the isometry $h$ to act as a permutation on the standard basis $\{e_1,\ldots,e_n\}$.
\begin{thm} \label{ncyclethm} Let $L=\mathbb{Z}^n$ and choose $h$ to act as an $n$-cycle on the standard basis. Then the resulting lattice $K$ is isometric to $A_{2n-1}$.\end{thm}
\begin{proof}
Since any two $n$-cycles are conjugate, we may assume that $h$ is the $n$-cycle $\{e_1,e_2,\ldots,e_n\}$.  We choose the orthonormal basis $\{e_1,\ldots, e_n, -e_{n+1}, \ldots, -e_{2n}\}$ for $\mathbb{Z}^n \perp \mathbb{Z}^n$, so the lattice $M$ is spanned by the vectors $\{e_1-e_{n+1},\ldots,e_n-e_{2n}\}$ and $N$ is spanned by $\{e_1-e_{n+2},\ldots,e_{n-1}-e_{2n},e_n-e_{n+1}\}$.  Therefore, all vectors in $K:=M+N$ have zero coordinate sum, so $K\leq A_{2n-1}$.  We can verify that the vectors $\{e_1-e_{n+1},\ldots,e_n-e_{2n}, e_1-e_{n+2},\ldots,e_{n-1}-e_{2n},e_n-e_{n+1}\}$ in fact span all of $A_{2n-1}$.  For $1\leq i \leq n-1$, we can write $e_i-e_{i+1}=(e_i-e_{i+n+1})-(e_{i+1}-e_{i+n+1})$.  For $n+1\leq i \leq 2n-1$, we can write $e_i-e_{i+1}=(e_{i-n}-e_{i+1})-(e_{i-n}-e_{i})$.  Since we are already given the vector $e_n-e_{n+1}$, we have accounted for all vectors of the form $e_i-e_{i+1}, 1 \leq i \leq 2n-1$, and these vectors clearly span the whole lattice $A_{2n-1}$.  Hence, $K\cong A_{2n-1}$.  
\end{proof}
Given any permutation $h$, we can express $h$ as a product of disjoint cycles.  We may then express $\mathbb{Z}^n$ as an orthogonal sum of sublattices, each isometric to $\mathbb{Z}^m$ for some $m\leq n$, on which $h$ acts as an $m$-cycle.  The above theorem then gives us the corresponding sublattice of $K$ for each of the disjoint cycles, and the whole lattice $K$ can then be represented as the orthogonal sum of these sublattices.

It is natural to extend this characterization to the even sublattice $D_n<\mathbb{Z}^n$.  The following theorem gives the index of the sublattice $K_{D_n}<K_{\mathbb{Z}^n}$ in the case where $h$ is an $n$-cycle:
\begin{thm}\label{DnSublattice}Let $L=\mathbb{Z}^n$, and let $L'\leq L$ be the even sublattice $D_n$.  Then the lattice $K$ formed by applying an $n$-cycle to $L$ contains the lattice $K'$ formed by applying the same $n$-cycle to $L'$ with index 2 for odd $n$ and index 4 for even $n$.\end{thm}
\begin{proof}  Take $h$ to be an $n$-cycle.  We can write each vector in $K$ as a sum $(v,v)+(w,hw)$ for $v, w \in L$.  Such a vector is also in $K'$ if we can find such an expression where $v$ and $w$ each have even coordinate sum.  Clearly, if one of these vectors has even coordinate sum and the other has odd coordinate sum, the resulting vector is not contained in $K'$ since $v+w$ has odd coordinate sum.  This gives two possible cosets of $K'$ in $K$, and the trivial case $K'$ and the case of $v$ and $w$ both having odd coordinate sum complete the four possibilites.

For odd $n$, we can show that there are only two cosets as follows.  For each pair of vectors $v,w$, we can define $v'=v+(1,1,\ldots,1)$ and $w'=w-(1,1,\ldots,1)$.  Since $h$ fixes $(1,1,\ldots,1)$, we have $(v',v')+(w',hw')=(v,v)+(w,hw)$.  But $v'$ and $w'$ have the opposite coordinate sum parity of $v$ and $w$, respectively, so this shows that the cases of exactly one of $v,w$ having odd coordinate sum are equivalent, and that the case of $v$ and $w$ both having odd coordinate sum is equivalent to the case where both have even coordinate sum.  Hence, there are only two cosets and $K$ contains $K'$ with index 2 for odd $n$. 

However, if $n$ is even, the four cosets are distinct.  To show this, suppose first that $v$ and $w$ both have odd coordinate sum.  We show that $(v,v)+(w,hw)\not\in K'$.  Suppose that $v'$ and $w'$ are vectors with even coordinate sum such that $(v',v')+(w',hw')=(v,v)+(w,hw)$.  Then we have $(v'-v+w'-w,v'-v+hw'-hw)=0$, so from the left side $v'-v=w-w'$.  But the right side gives $v'-v=hw-hw'=h(w-w')$.  Thus $w-w'$ is fixed by $h$.  Since $h$ is an $n$-cycle, $w-w'$ must be a multiple of $(1,1,\ldots,1)$.  But such a vector always has even coordinate sum, while $w-w'$ has odd coordinate sum, giving a contradiction.  The same argument shows that the two cases with exactly one of $v,w$ having odd coordinate sum are distinct as well, so there are four cosets and $K$ contains $K'$ with index 4 for even $n$.
\end{proof}
The proof for the case where $n$ is odd does not require $h$ to be an $n$-cycle, but only that $h$ fixes $(1,1,\ldots,1)$.  Hence, we can apply this part of the theorem for any permutation $h$.  The even case, however, relies on the fact that any vector fixed by $h$ has even coordinate sum, and hence requires that the cycle structure of $h$ consist only of even cycles.

The above theorem describes the lattices obtained from $D_n$ by permutations, and the following theorem covers the cases where $h$ negates some coordinates.
\begin{thm}\label{DnNegations}Let $L=D_n$, and let $h$ negate $k$ coordinates and fix the remaining coordinates, where $k<n$.  Then the resulting lattice $K\cong\sqrt{2}D_{n+k}$. \end{thm}
\begin{proof}
We may assume $h$ negates the first $k$ coordinates.  The lattice $K_{\mathbb{Z}^n}$ obtained from $\mathbb{Z}^n$ by applying $h$ has the orthogonal basis $\{u_1,\ldots,u_{n+k}\}$, where $u_1=e_1+e_{n+1},\ldots,u_n=e_n+e_{2n},u_{n+1}=e_1-e_{n+1},\ldots,u_{n+k}=e_k-e_{n+k}$, and thus $K_{\mathbb{Z}^n}$ is isometric to ${A_1}^{n+k}\cong\sqrt{2}\mathbb{Z}^{n+k}$.  In the same way that the vectors $\pm e_i\pm e_j$ span $D_n<\mathbb{Z}^n$, the vectors of the form $\pm u_i\pm u_j$ span a sublattice of $K_{\mathbb{Z}^n}$ isometric to $\sqrt{2}D_{n+k}$.

We now show that all vectors of this form are contained in $K_{D_n}=M_{D_n}+N_{D_n}$.  If $i,j\leq n$, we have $\pm u_i\pm u_j=(\pm e_i\pm e_j)+(\pm e_{n+i}\pm e_{n+j})=(v,v)\in M_{D_n}$, where we take $v=\pm e_i \pm e_j \in D_n$ with the same signs as $\pm u_i\pm u_j$.  If $i,j> n$, we have $\pm u_i\pm u_j=(\pm e_{i-n}\pm e_{j-n})-(\pm e_i\pm e_j)=(v,hv)\in N_{D_n}$, where we take $v=\pm e_{i-n} \pm e_{j-n} \in D_n$ with the same signs as $\pm u_i\pm u_j$.  If $i\leq n$ and $j>n$, we have $\pm u_i\pm u_j=\pm((e_i+e_n)+(e_{n+i}+e_{2n}))\pm ((e_{j-n}\pm e_n)+(-e_j\pm e_{2n}))=(v,v)+(w,hw)\in M_{D_n}+N_{D_n}=K_{D_n}$, where we take $v=\pm(e_i + e_n)$ with the same sign as $u_i$, and $w=\pm e_{j-n} \pm e_n$, where the sign of the $e_{j-n}$ term is the same as the sign of $u_j$ and the sign of the $e_n$ is the opposite of the sign of $u_i$.  Since $k<n$, $h$ fixes $e_n$, so with this choice of sign the $e_n$ and $e_{2n}$ terms cancel.

This covers all the cases, so we have $\sqrt{2}D_{n+k}\leq K_{D_n}$.  $K_{D_n}$ cannot be the whole lattice $K_{\mathbb{Z}^n}={A_1}^{n+k}$, since all vectors in $K_{D_n}$ must have even sum in the first $n$ coordinates, while $K_{\mathbb{Z}^n}$ contains vectors such as the $u_i$ that have odd sum in the first $n$ coordinates.  Thus, since $\sqrt{2}D_{n+k}$ is a sublattice of index 2, we have $K_{D_n}\cong\sqrt{2}D_{n+k}$.
\end{proof}

In the following result we refer to the tensor product of lattices.  Given lattices $L_1,L_2$, the tensor product $L_1\otimes L_2$ is the lattice whose underlying free abelian group is $L_1 \otimes_{\mathbb{Z}} L_2$, and whose inner product is given by
\begin{equation*}
(x_1\otimes x_2,y_1\otimes y_2)=(x_1,y_1)\cdot (x_2,y_2).
\end{equation*}
 The Gram matrix of such a tensor product is the Kronecker product of the Gram matrices of $L_1$ and $L_2$.
\begin{thm}\label{A2Tensor}  Let $L$ be any lattice, and let $h\in O(L)$ be an element satisfying $h^2+h+1=0$.  Then the resulting lattice $K\cong A_2\otimes L$. \end{thm}
\begin{proof}  Let $\{v_1,\ldots,v_n\}$ be a basis for $L$ with Gram matrix $G$.  We choose the standard basis $\{u_1,\ldots,u_n\}$ for $M$, where $u_i=(v_i,v_i), 1\leq i \leq n$, and the basis $\{w_1,\ldots,w_n\}$ for $N$, where $w_i=(hv_i,h^2v_i), 1\leq i \leq n$.  This is indeed a basis for $N$, since $\{hv_1,\ldots,hv_n\}$ is a basis for $L$ because $h$ is an isometry.  Clearly, the inner products $(u_i,u_j)=(w_i,w_j)=2(v_i,v_j), 1\leq i,j\leq n$, since $h$ preserves inner products.  Also, 
\begin{equation*}
(u_i,w_j)=(v_i,hv_j)+(v_i,h^2v_j)=(v_i,hv_j+h^2v_j)=(v_i,-v_j)=-(v_i,v_j), \quad1\leq i,j \leq n,
\end{equation*}
so the Gram matrix for $K:=M+N$ with respect to the basis $\{u_1,\ldots,u_n,w_1,\ldots,w_n\}$ has the block form $\begin{pmatrix}2G&-G\\-G&2G\end{pmatrix}=\begin{pmatrix}2&-1\\-1&2\end{pmatrix}\otimes G$, which is the Gram matrix of $A_2 \otimes L$.
\end{proof}

\section{$L=\mathbb{Z}^2$}
In the following sections, we take various choices of $L$ and compute $K$ for all possible conjugacy classes.  We begin with $\mathbb{Z}^2$.

The isometry group $O(\mathbb{Z}^2)$ is dihedral of order 8.  There are five conjugacy classes. The cases $h=\pm 1$ are discussed in the introduction.  The remaining three cases follow.

\subsection{Reflection Across a Coordinate Axis}
The reflections across the $x$- and $y$-axes are conjugate in the isometry group.  We choose $h$ to be the reflection across the $x$-axis, so $h(1,0)=(1,0)$ and $h(0,1)=(0,-1)$.  Thus, the lattice $K$ is spanned by the vectors $\{ (1,0,1,0),(0,1,0,1),(0,1,0,-1) \}$. These vectors are pairwise orthogonal and all have norm 2, so $K$ is the rectangular lattice $A_{1}^3$.

\subsection{Reflection Across a Diagonal Axis}
The reflections across the lines $y=x$ and $y=-x$ are conjugate.  Let $h$ be the reflection across $y=x$, so $h(1,0)=(0,1)$ and $h(0,1)=(1,0)$.  Note that this represents a transposition on the orthonormal basis $\{e_1,e_2\}$, so by Theorem \ref{ncyclethm} $K\cong A_3$.

\subsection{Rotation by $90^{\circ}$}
\label{Z2-90}
The clockwise and counterclockwise $90^{\circ}$ rotations are conjugate.  Let $h$ be the counterclockwise rotation, so $h(1,0)=(0,1)$ and $h(0,1)=(-1,0)$.  Thus, $K$ is spanned by $\{ (1,0,1,0),(0,1,0,1),(1,0,0,1),(0,1,-1,0)\}$.  The Gram matrix
\begin{equation*}
G_K=
\begin{pmatrix}
2 & 0 & 1 & -1 \\
0 & 2 & 1 & 1 \\
1 & 1 & 2 & 0 \\
-1 & 1 & 0 & 2
\end{pmatrix}
\end{equation*}
has determinant 4.  Since any even lattice of rank 4 and determinant 4 is isometric to $D_4$ (Theorem \ref{Rank4}), $K \cong D_4$.

\section{$L=A_2$}
The isometry group $O(A_2)$ is dihedral of order 12, and there are six conjugacy classes.  Again, the cases $h=\pm 1$ are simple and covered in the introduction.  We take the standard basis $\{ v_1, v_2 \}$ of $L$, where the inner products $(v_1,v_1)=(v_2,v_2)=2$ and $(v_1,v_2)=-1$.  These vectors may be represented in $\mathbb{R}^3$ as $v_1=e_1-e_2,v_2=e_2-e_3$, or in $\mathbb{R}^2$ as $v_1=(\sqrt{2},0),v_2=(-\frac{\sqrt{2}}{2},\frac{\sqrt{6}}{2})$.  The latter representation is useful in visualizing the lattice and its isometry group.

\subsection{Reflection Across an Axis Through a Root}
The reflections across the three lines through a root are conjugate.  We take $h$ to be the reflection across the line through $v_1$, so $hv_1=v_1$ and $hv_2=-v_1-v_2$.  Then $K$ is spanned by $\{ (v_1,v_1),(v_2,v_2),(v_2,-v_1-v_2)\} $.  The Gram matrix
\begin{equation*}
G_K=
\begin{pmatrix}
4 & -2 & -2 \\
-2 & 4 & 1 \\
-2 & 1 & 4 
\end{pmatrix}
\end{equation*}
has determinant 36.  The vector $w:=2(v_2, -v_1-v_2)+(v_1,v_1)=(v_1+2v_2,-v_1-2v_2)$ is orthogonal to $M=span((v_1,v_1),(v_2,v_2))$, so $K$ contains a sublattice $J\cong \sqrt{2}A_2 \perp \mathbb{Z}w$.  The determinant of $J$ is $12\cdot 12=144$, so $J$ has index 2 in $K$.  We can form $K$ from $J$ by means of the gluing $K=J+\mathbb{Z}(v_2,-v_1-v_2)$.

Also, the Smith invariant sequence for $G_K$ is 1, 3, 12, so there exists $u \in K^*$ such that $K + u$ has order 4 in the discriminant group.  Then $(2u,2u)=(4u,u)\in \mathbb{Z}$, so $K+2\mathbb{Z} u$ is an integral lattice.  Since this lattice contains $K$ as a sublattice of index 2, it has determinant 9.

\subsection{Reflection Across an Axis Perpendicular to a Root}
The three reflections across lines perpendicular to roots are conjugate.  We take $h$ to be the reflection $r_{v_1}$ across a line perpendicular to $v_1$, so $hv_1=-v_1$ and $hv_2=v_1+v_2$.  Then $K$ is spanned by $\{ (v_1,v_1),(v_2,v_2),(v_1,-v_1),(v_2,v_1+v_2)\} $.  Since $(v_2,v_1+v_2)-(v_2,v_2)=(v_1,0)$ and $(v_1,v_1)-(v_1,0)=(0,v_1)$, we have $K=span((v_1,0),(0,v_1),(v_2,v_2))$.  This gives the Gram matrix
\begin{equation*}
G_K=
\begin{pmatrix}
2 & 0 & -1 \\
0 & 2 & -1 \\
-1 & -1 & 4 
\end{pmatrix},
\end{equation*}
which has determinant 12 and Smith invariant sequence 1, 1, 12.  The vector $(v_1,-v_1)$ is orthogonal to each of the vectors $(v_1,v_1), (v_2,v_2)$, so we can construct a sublattice $J\cong \sqrt{2}A_2 \perp \mathbb{Z}(v_1,-v_1)$, which has determinant $12\cdot 4=48$ and is thus a sublattice of index 2 in $K$ by Theorem \ref{IndexDet}.

Since the discriminant group $\mathcal{D}(K)$ is cyclic of order 12, there is a vector $u \in K^*$ with order $4 \pmod{K}$, so, as above, we can embed $K$ with index 2 in the integral lattice $H:=K+2\mathbb{Z}u$, which has determinant 3.  $H$ is therefore either rectangular or isometric to $\mathbb{Z}\perp A_2$ (Theorem \ref{Rank3}).  Thus $H$ is not even, so $K$ must be the even sublattice of $H$ since $K$ is even and a sublattice of index 2 in $H$.  Suppose $H$ is rectangular.  Then $H$ has a basis $\{u_1,u_2,u_3\}$ whose Gram matrix is
\begin{equation*}
G_H=
\begin{pmatrix}
1 & 0 & 0 \\
0 & 1 & 0 \\
0 & 0 & 3 
\end{pmatrix}.
\end{equation*}
The even sublattice of $H$ is then spanned by $\{u_1-u_2,u_1+u_2,u_3-u_1\}$, and the Gram matrix for the even sublattice with respect to this basis is
\begin{equation*}
\begin{pmatrix}
2 & 0 & -1 \\
0 & 2 & -1 \\
-1 & -1 & 4 
\end{pmatrix}=G_K.
\end{equation*}
Thus, we can construct $K$ as the even sublattice of the rectangular lattice $\mathbb{Z}^2 \perp \sqrt{3} \mathbb{Z}$.

\subsection{Rotation by $60^{\circ}$}
The clockwise and counterclockwise $60^{\circ}$ rotations are conjugate.  We take $h$ to be the counterclockwise rotation, so $hv_1=v_1+v_2$ and $hv_2=-v_1$.  Then $K$ is spanned by $\{(v_1,v_1),(v_2,v_2),(v_1,v_1+v_2),(v_2,-v_1)\}$.  We have
\begin{eqnarray}
(v_1,v_1+v_2)-(v_1,v_1)=(0,v_2) \nonumber \\
(v_2,v_2)-(0,v_2)=(v_2,0) \nonumber \\
(v_2,0)-(v_2,-v_1)=(0,v_1) \nonumber \\
(v_1,v_1)-(0,v_1)=(v_1,0) \nonumber
\end{eqnarray}
and these vectors span the whole lattice $L\perp L$.

\subsection{Rotation by $120^{\circ}$}
The clockwise and counterclockwise $120^{\circ}$ rotations are conjugate.  In either case, if we take $h$ to be a $120^\circ$ rotation, we have $h^2+h+1=0$, so by Theorem \ref{A2Tensor} we have $K\cong A_2 \otimes A_2$.

\section{$L=A_3$}
The isometry group of $A_3$ is $O(A_3)=Sym_4\times \langle -1 \rangle$, where we take the standard basis $\{v_1=e_1-e_2,v_2=e_2-e_3,v_3=e_3-e_4\} \subset \mathbb{R}^4$ for $A_3$ and $Sym_4$ acts on the set $\{e_1,e_2,e_3,e_4\}$.  There are 10 conjugacy classes, two (corresponding to a choice of $\pm 1$) for each partition of four elements into cycles.  The cases $h=\pm 1$ are trivial and discussed in the introduction.

\subsection{Transposition}\label{A3transp}
There are 6 elements of $O(A_3)$ that correspond to a transposition of two unit vectors $e_i,e_j$.  We take $h$ to be the transposition $(e_1,e_2)=r_{v_1}$, the reflection at the root $v_1=e_1-e_2$.  We have $hv_1=h(e_1-e_2)=e_2-e_1=-v_1$, $hv_2=h(e_2-e_3)=e_1-e_3=v_1+v_2$, and $hv_3=h(e_3-e_4)=e_3-e_4=v_3$.  Then $K$ is spanned by $\{ (v_1,v_1),(v_2,v_2),(v_3,v_3),(v_1,-v_1),\\(v_2,v_1+v_2)\}$.  We have $(v_2,v_1+v_2)-(v_2,v_2)=(0,v_1)$ and $(v_1,v_1)-(0,v_1)=(v_1,0)$, and hence we see that $K=span((v_1,0),(0,v_1),(v_2,v_2),(v_3,v_3))$.  The Gram matrix is
\begin{equation*}
G_K=
\begin{pmatrix}
2 & 0 & -1 & 0 \\
0 & 2 & -1 & 0 \\
-1 & -1 & 4 & -2 \\
0 & 0 & -2 & 4
\end{pmatrix},
\end{equation*}
which has determinant 32 and Smith invariants 1, 1, 4, 8.  Thus, the discriminant group $\mathcal{D}(K)\cong \mathbb{Z}/4\mathbb{Z} \oplus \mathbb{Z}/8\mathbb{Z}$, so we can choose independent vectors $u$ and $w$ in $K^*$, both of order $4 \pmod{K}$.  Since $(2u,2u)=(4u,u) \in \mathbb{Z}$, $(2w,2w)=(4w,w) \in \mathbb{Z}$, and $(2u,2w)=(4u,w) \in \mathbb{Z}$, the lattice $H:=K+2\mathbb{Z}u+2\mathbb{Z}w$ is integral, contains $K$ as a sublattice of index 4, and thus has determinant 2 (Theorem \ref{IndexDet}).  Using the Hermite function $H(4,2)=1.8309\dots$ (Theorem \ref{Hermite}), $H$ contains a unit vector $x$, so we can decompose orthogonally, since for $v\in H$ we have $v=(v,x)x+(v-(v,x)x)\in \mathbb{Z}x \perp ann_H(x)$.  Hence $H\cong \mathbb{Z}x \perp G$ for an integral lattice $G$ of rank 3 and determinant 2.  Such a lattice must be rectangular (Theorem \ref{Rank3}), so $H$ is the rectangular lattice $\mathbb{Z}^3 \perp \sqrt{2}\mathbb{Z}$.

Also, we can see from our construction that $K$ contains a sublattice $J\cong \sqrt{2}A_3 \perp \mathbb{Z}(v_1,-v_1)$, since $M=span((v_1,v_1),(v_2,v_2),(v_3,v_3))\cong \sqrt{2}A_3$, and  $(v_1,-v_1)\in ann_K(M)$. Since $\sqrt{2}A_3$ has determinant $2^3\cdot 4=32$ and $(v_1,-v_1)$ has norm 4, $J$ has determinant 128 and thus is a sublattice of index 2 in $K$.  We can form $K$ from $J$ by means of the ``glue vector'' $(v_1,0)$: $K=J+\mathbb{Z}(v_1,0)$.

\subsection{3-Cycle}
There are 8 elements of $O(A_3)$ that correspond to a 3-cycle of unit vectors $e_i,e_j,e_k$.  We take $h$ to be the cycle $(e_1,e_2,e_3)=r_{v_1}r_{v_2}$.  We have $hv_1=v_2$, $hv_2=-v_1-v_2$, and $hv_3=v_1+v_2+v_3$.  Then $K$ is spanned by $\{ (v_1,v_1),(v_2,v_2),(v_3,v_3),(v_1,-v_2),\\(v_2,-v_1-v_2),(v_3,v_1+v_2+v_3)\}$.  We have 
\begin{eqnarray}
(v_2,-v_1-v_2)+(v_3,v_1+v_2+v_3)-(v_3,v_3)=(v_2,0) \nonumber \\
(v_2,v_2)-(v_2,0)=(0,v_2) \nonumber \\
(v_1,v_2)-(0,v_2)=(v_1,0) \nonumber \\
(v_1,v_1)-(v_1,0)=(0,v_1) \nonumber
\end{eqnarray}
and thus $K=span((v_1,0),(0,v_1),(v_2,0),(0,v_2),(v_3,v_3))$.  We obtain the Gram matrix

\begin{equation*}
G_K=
\begin{pmatrix}
2 & 0 & -1 & 0 & 0 \\
0 & 2 & 0 & -1 & 0 \\
-1 & 0 & 2 & 0 & -1 \\
0 & -1 & 0 & 2 & -1 \\
0 & 0 & -1 & -1 & 4 \\
\end{pmatrix},
\end{equation*}
which has determinant 24 and Smith invariants 1, 1, 1, 1, 24.  Using the same procedure as above (choosing a vector of order $4 \pmod{K}$), we may embed $K$ with index 2 in an integral lattice of determinant 6.  We can also construct a sublattice $J\cong \sqrt{2}A_3 \perp \sqrt{2}A_2$, where we identify $M$ with $\sqrt{2}A_3$ as before, and $span((v_1,-v_1),(v_2,-v_2))$ with $\sqrt{2}A_2$.  This is a sublattice of index 4 in $K=J+\mathbb{Z}(v_1,0)+\mathbb{Z}(v_2,0)$.

\subsection{Product of Disjoint Transpositions}
There are 3 elements of $O(A_3)$ that correspond to a product of disjoint transpositions.  We take $h=(e_1,e_2)(e_3,e_4)=r_{v_1}r_{v_3}$, so $hv_1=-v_1$, $hv_2=v_1+v_2+v_3$, and $hv_3=-v_3$.  Then $K$ is spanned by $\{ (v_1,v_1),(v_2,v_2),(v_3,v_3),(v_1,-v_1),(v_2,v_1+v_2+v_3),(v_3,-v_3)\}$.  We can write the last vector as a linear combination of the others:
\begin{equation*}
(v_3,-v_3)=(v_1,v_1)+2(v_2,v_2)+(v_3,v_3)-(v_1,-v_1)-2(v_2,v_1+v_2+v_3).
\end{equation*}
The remaining 5 vectors are linearly independent and give the Gram matrix
\begin{equation*}
G_K=
\begin{pmatrix}
4 & -2 & 0 & 0 & 0 \\
-2 & 4 & -2 & 0 & 2 \\
0 & -2 & 4 & 0 & 0 \\
0 & 0 & 0 & 4 & -2 \\
0 & 2 & 0 & -2 & 4
\end{pmatrix},
\end{equation*}
which has determinant 128 and Smith invariants 2, 2, 2, 2, 8.  Since all entries in the Gram matrix are even, the lattice $J:=K/\sqrt{2}$ is an even integral lattice with determinant 4 and cyclic discriminant group.  Thus, for $u \in J^*$, we have $(2u,2u)=(4u,u)\in \mathbb{Z}$, so $J+2J^*$ is an integral lattice which contains $J$ with index 2 and thus has determinant 1.  Therefore, $J+2J^* \cong \mathbb{Z}^5$ (Theorem \ref{Unimodular}), and since $J$ is an even sublattice of index 2, $J$ is isometric to the even sublattice of $\mathbb{Z}^5$, which is $D_5$.  Hence, $K \cong \sqrt{2}D_5$.

\subsection{4-Cycle}
There are 6 elements of $O(A_3)$ that correspond to 4-cycles.  We take $h=(e_1,e_2, e_3,e_4)=r_{v_1}r{v_2}r_{v_3}$, so $hv_1=v_2$, $hv_2=v_3$, and $hv_3=-v_1-v_2-v_3$.  Then $K$ is spanned by $\{ (v_1,v_1),(v_2,v_2),(v_3,v_3),(v_1,v_2),(v_2,v_3),(v_3,-v_1-v_2-v_3)\}$.  The Gram matrix is
\begin{equation*}
G_K=
\begin{pmatrix}
4 & -2 & 0 & 1 & -1 & -1 \\
-2 & 4 & -2 & 1 & 1 & -1 \\
0 & -2 & 4 & -1 & 1 & 1 \\
1 & 1 & -1 & 4 & -2 & 0\\
-1 & 1 & 1 & -2 & 4 & -2 \\
-1 & -1 & 1 & 0 & -2 & 4
\end{pmatrix},
\end{equation*}
which has determinant 256.  Hence, $K$ is a sublattice of index 4 in $L \perp L$ by Theorem \ref{IndexDet}.

\subsection{Negative of transposition}
There are 6 elements of $O(A_3)$ that correspond to the negative of a transposition.  We take $h=-r_{v_1}$.  We have $hv_1=v_1$, $hv_2=-v_1-v_2$, and $hv_3=-v_3$.  Then $K$ is spanned by $\{ (v_1,v_1),(v_2,v_2),(v_3,v_3),(v_2,-v_1-v_2),(v_3,-v_3)\}$.  We obtain the Gram matrix
\begin{equation*}
G_K=
\begin{pmatrix}
4 & -2 & 0 & -2 & 0 \\
-2 & 4 & -2 & 1 & 0 \\
0 & -2 & 4 & 0 & 0 \\
-2 & 1 & 0 & 4 & -2 \\
0 & 0 & 0 & -2 & 4
\end{pmatrix},
\end{equation*}
which has determinant 256 and Smith invariants 1, 1, 4, 8, 8.  Using the process we used in \ref{A3transp} and choosing three independent vectors of order 4 in the dual, we can embed $K$ with index 8 in an integral lattice of determinant 4.

\subsection{Negative of 3-Cycle}
There are 6 elements of $O(A_3)$ that correspond to the negative of a 3-cycle.  We take $h=-r_{v_1}r_{v_2}$.  We have $hv_1=-v_2$, $hv_2=v_1+v_2$, and $hv_3=-v_1-v_2-v_3$.  Then $K$ is spanned by $\{ (v_1,v_1),(v_2,v_2),(v_3,v_3),(v_1,-v_2),(v_2,v_1+v_2),(v_3,-v_1-v_2-v_3)\}$.  We have
\begin{eqnarray}
(v_2,v_1+v_2)-(v_2,v_2)=(0,v_1) \nonumber \\
(v_1,v_1)-(0,v_1)=(v_1,0) \nonumber \\
(v_1,0)-(v_1,-v_2)=(0,v_2) \nonumber \\
(v_2,v_2)-(0,v_2)=(v_2,0) \nonumber \\
(v_3,-v_1-v_2-v_3)+(0,v_1)+(0,v_2)=(v_3,-v_3) \nonumber
\end{eqnarray}
and $K$ is thus spanned by $\{(v_1,0),(0,v_1),(v_2,0),(0,v_2),(v_3,v_3),(v_3,-v_3)\}$.  The Gram matrix
\begin{equation*}
G_K=
\begin{pmatrix}
2 & 0 & -1 & 0 & 0 & 0 \\
0 & 2 & 0 & -1 & 0 & 0 \\
-1 & 0 & 2 & 0 & -1 & -1 \\
0 & -1 & 0 & 2 & -1 & 1\\
0 & 0 & -1 & -1 & 4 & 0 \\
0 & 0 & -1 & 1 & 0 & 4
\end{pmatrix}
\end{equation*}
has determinant 64, so $K$ is a sublattice of index 2 in $L \perp L$.

\subsection{Negative of Product of Disjoint Transpositions}
There are 3 elements of $O(A_3)$ that correspond to negatives of products of disjoint transpositions.  We take $h=-r_{v_1}r_{v_3}$, so $hv_1=v_1$, $hv_2=-v_1-v_2-v_3$, and $hv_3=v_3$.  Then $K$ is spanned by $\{ (v_1,v_1),(v_2,v_2),(v_3,v_3),(v_2,-v_1-v_2-v_3)\}$.
The Gram matrix is 
\begin{equation*}
G_K=
\begin{pmatrix}
4 & -2 & 0 & -2 \\
-2 & 4 & -2 & -2 \\
0 & -2 & 4 & -2 \\
-2 & -2 & -2 & 4
\end{pmatrix},
\end{equation*}
which has determinant 64.  Note that $K/\sqrt{2}$ is an even integral lattice of rank 4 and determinant 4, and is thus isometric to $D_4$ (Theorem \ref{Rank4}).  We therefore have $K \cong \sqrt{2}D_4$.

\subsection{Negative of 4-Cycle}
There are 6 elements of $O(A_3)$ that correspond to negatives of 4-cycles.  We take $h=-r_{v_1}r{v_2}r_{v_3}$, so $hv_1=-v_2$, $hv_2=-v_3$, and $hv_3=v_1+v_2+v_3$.  Then $K$ is spanned by $\{ (v_1,v_1),(v_2,v_2),(v_3,v_3),(v_1,-v_2),(v_2-,v_3),(v_3,v_1+v_2+v_3)\}$. We can write the last vector as a linear combination of the others:
\begin{equation*}
(v_3,v_1+v_2+v_3)=(v_1,v_1)+(v_3,v_3)-(v_1,-v_2).
\end{equation*}
The remaining vectors are linearly independent and give the Gram matrix
\begin{equation*}
G_K=
\begin{pmatrix}
4 & -2 & 0 & 3 & -1 \\
-2 & 4 & -2 & -3 & 3 \\
0 & -2 & 4 & 1 & -3 \\
-3 & -3 & 1 & 4 & -2 \\
-1 & 3 & -3 & -2 & 4
\end{pmatrix},
\end{equation*}
which has determinant 32 and Smith invariants 1, 1, 2, 2, 8.

\section{$L=A_4$}
The isometry group of $A_4$ is $O(A_4)=Sym_5\times \langle -1 \rangle$, where we take the standard basis $\{v_1=e_1-e_2,v_2=e_2-e_3,v_3=e_3-e_4,v_4=e_4-e_5\} \subset \mathbb{R}^5$ for $A_4$ and $Sym_5$ acts on the set $\{e_1,e_2,e_3,e_4,e_5\}$.  There are 14 conjugacy classes, two (corresponding to a choice of $\pm 1$) for each partition of five elements into cycles.  The cases $h=\pm 1$ have been discussed.

\subsection{Transposition}
There are 10 elements of $O(A_4)$ that correspond to a transposition of two unit vectors $e_i,e_j$.  We take $h$ to be the transposition $(e_1,e_2)=r_{v_1}$, the reflection at the root $v_1$.  We have $hv_1=-v_1$, $hv_2=v_1+v_2$,  $hv_3=v_3$, and $hv_4=v_4$.  Then $K$ is spanned by $\{ (v_1,v_1),(v_2,v_2),(v_3,v_3),(v_4,v_4),(v_1,-v_1),(v_2,v_1+v_2)\}$.  Since $(v_2, v_1+v_2)-(v_2,v_2)=(0,v_1)$ and $(v_1,v_1)-(0,v_1)=(v_1,0)$, $K$ is spanned by $\{(v_1,0),(0,v_1),(v_2,v_2),(v_3,v_3),(v_4,v_4)\}$.  The Gram matrix with respect to this basis is 
\begin{equation*}
G_K=
\begin{pmatrix}
2 & 0 & -1 & 0 & 0 \\
0 & 2 & -1 & 0 & 0 \\
-1 & -1 & 4 & -2 & 0 \\
0 & 0 & -2 & 4 & -2 \\
0 & 0 & 0 & -2 & 4
\end{pmatrix},
\end{equation*}
which has determinant 80 and Smith invariants 1, 1, 2, 2, 20.  $K$ contains a sublattice $J\cong\sqrt{2}A_4 \perp \mathbb{Z}(v_1,-v_1)$ with index 2, and $K=J+\mathbb{Z}(v_1,0)$.

\subsection{Product of Disjoint Transpositions}
There are 15 elements of $O(A_4)$ that correspond to a product of disjoint transpositions.  We take $h=(e_1,e_2)(e_3,e_4)=r_{v_1}r_{v_3}$, and we have $hv_1=-v_1$, $hv_2=v_1+v_2+v_3$,  $hv_3=-v_3$, and $hv_4=v_3+v_4$.  Then $K$ is spanned by $\{ (v_1,v_1),(v_2,v_2),(v_3,v_3),(v_4,v_4),(v_1,-v_1),\\(v_2,v_1+v_2+v_3),(v_3,-v_3),(v_4,v_3+v_4)\}$.  We have
\begin{eqnarray}
(v_4,v_3+v_4)-(v_4,v_4)=(0,v_3) \nonumber \\
(v_3,v_3)-(0,v_3)=(v_3,0) \nonumber \\
(v_2,v_1+v_2+v_3)-(0,v_3)-(v_2,v_2)=(0,v_1) \nonumber \\
(v_1,v_1)-(0,v_1) =(v_1,0) \nonumber
\end{eqnarray}
and thus $K=span((v_1,0),(0,v_1),(v_2,v_2),(v_3,0),(0,v_3),(v_4,v_4))$.  The Gram matrix
\begin{equation*}
G_K=
\begin{pmatrix}
2 & 0 & -1 & 0 & 0 & 0 \\
0 & 2 & -1 & 0 & 0  & 0\\
-1 & -1 & 4 & -1 & -1 & 0 \\
0 & 0 & -1 & 2 & 0 & -1 \\
0 & 0 & -1 & 0 & 2 & -1 \\
0 & 0 & 0 & -1 & -1 & 4
\end{pmatrix}
\end{equation*}
has determinant 80 and Smith invariants 1, 1, 1, 1, 4, 20.   $K$ contains a sublattice $J\cong\sqrt{2}A_4 \perp \mathbb{Z}(v_1,-v_1) \perp \mathbb{Z}(v_3,-v_3)$ with index 4, and $K=J+\mathbb{Z}(v_1,0)+\mathbb{Z}(v_3,0)$.

\subsection{3-Cycle}
There are 20 elements of $O(A_4)$ that correspond to 3-cycles.  We take $h=(e_1,e_2,e_3)=r_{v_1}r_{v_2}$, and we have $hv_1=v_2$, $hv_2=-v_1-v_2$,  $hv_3=v_1+v_2+v_3$, and $hv_4=v_4$.  Then $K$ is spanned by $\{ (v_1,v_1),(v_2,v_2),(v_3,v_3),(v_4,v_4),(v_1,v_2),(v_2,-v_1-v_2),(v_3,v_1+v_2+v_3)\}$.  We have
\begin{eqnarray}
(v_2,-v_1-v_2)+(v_3,v_1+v_2+v_3)-(v_3,v_3)=(v_2,0) \nonumber \\
(v_2,v_2)-(v_2,0)=(0,v_2) \nonumber \\
(v_1,v_2)-(0,v_2)=(v_1,0) \nonumber \\
(v_1,v_1)-(v_1,0) =(0,v_1) \nonumber
\end{eqnarray}
and thus $K=span((v_1,0),(0,v_1),(v_2,0),(0,v_2),(v_3,v_3),(v_4,v_4))$.  The Gram matrix for this basis is 
\begin{equation*}
G_K=
\begin{pmatrix}
2 & 0 & -1 & 0 & 0 & 0 \\
0 & 2 & 0 & -1 & 0  & 0\\
-1 & 0 & 2 & 0 & -1 & 0 \\
0 & -1 & 0 & 2 & -1 & 0 \\
0 & 0 & -1 & -1 & 4 & -2 \\
0 & 0 & 0 & 0 & -2 & 4
\end{pmatrix},
\end{equation*}
which has determinant 60 and Smith invariants 1, 1, 1, 1, 2, 30.  We can identify $span((v_1,-v_1),(v_2,-v_2))$ with $\sqrt{2}A_2$, so $K$ contains a sublattice $J\cong\sqrt{2}A_4 \perp \sqrt{2}A_2$ with index 4, and $K=J+\mathbb{Z}(v_1,0)+\mathbb{Z}(v_2,0).$

\subsection{Product of 3-Cycle and Transposition}
There are 20 elements of $O(A_4)$ that correspond to products of disjoint 3-cycles and transpositions.  We take $h=(e_1,e_2,e_3)(e_4,e_5)=r_{v_1}r_{v_2}r_{v_4}$, and we have $hv_1=v_2$, $hv_2=-v_1-v_2$,  $hv_3=v_1+v_2+v_3+v_4$, and $hv_4=-v_4$.  Then $K$ is spanned by $\{ (v_1,v_1),(v_2,v_2),(v_3,v_3),(v_4,v_4),(v_1,v_2),(v_2,-v_1-v_2),(v_3,v_1+v_2+v_3+v_4),(v_4,-v_4)\}$.  We may write
\begin{align}
(v_1,v_2)+2(v_2,-v_1-v_2)+3(v_3,v_1+v_2+v_3+v_4)+(v_4,-v_4) &\nonumber \\
-\:(v_1,v_1)-2(v_2,v_2)-3(v_3,v_3)-(v_4,v_4)&=(0,v_4) \nonumber \\
(v_4,v_4)-(0,v_4)&=(v_4,0) \nonumber \\
(v_3,v_1+v_2+v_3+v_4)-(0,v_4)-(v_3,v_3)+(v_2,-v_1-v_2)&=(v_2,0) \nonumber \\
(v_2,v_2)-(v_2,0)&=(0,v_2) \nonumber \\
(v_1,v_2)-(0,v_2)&=(v_1,0) \nonumber \\
(v_1,v_1)-(v_1,0)&=(0,v_1) \nonumber
\end{align}
and thus $K=span((v_1,0),(0,v_1),(v_2,0),(0,v_2),(v_3,v_3),(v_4,0) ,(0,v_4))$.  The Gram matrix with respect to this basis is
\begin{equation*}
G_K=
\begin{pmatrix}
2 & 0 & -1 & 0 & 0 & 0 & 0 \\
0 & 2 & 0 & -1 & 0  & 0 & 0\\
-1 & 0 & 2 & 0 & -1 & 0 & 0 \\
0 & -1 & 0 & 2 & -1 & 0 & 0\\
0 & 0 & -1 & -1 & 4 & -1 & -1 \\
0 & 0 & 0 & 0 & -1 & 2 & 0 \\
0 & 0 & 0 & 0 & -1 & 0 & 2
\end{pmatrix},
\end{equation*}
which has determinant 60 and Smith invariants 1, 1, 1, 1, 1, 1, 60.  Again identifying $span((v_1,-v_1),(v_2,-v_2))$ with $\sqrt{2}A_2$, $K$ contains a sublattice $J\cong\sqrt{2}A_4 \perp \sqrt{2}A_2 \perp \mathbb{Z}(v_4,-v_4)$ with index 8, and $K=J+\mathbb{Z}(v_1,0)+\mathbb{Z}(v_2,0)+\mathbb{Z}(v_4,0)$.

\subsection{4-Cycle}
There are 30 elements of $O(A_4)$ that correspond to 4-cycles.  We take $h=(e_1,e_2,e_3,e_4)$, and we have $hv_1=v_2$, $hv_2=v_3$,  $hv_3=-v_1-v_2-v_3$, and $hv_4=v_1+v_2+v_3+v_4$.  Then $K$ is spanned by $\{ (v_1,v_1),(v_2,v_2),(v_3,v_3),(v_4,v_4),(v_1,v_2),(v_2,v_3),(v_3,-v_1-v_2-v_3),\\(v_4,v_1+v_2+v_3+v_4)\}$.  We may write
\begin{eqnarray}
(v_4,v_1+v_2+v_3+v_4)-(v_4,v_4)+(v_3,-v_1-v_2-v_3)=(v_3,0) \nonumber \\
(v_3,v_3)-(v_3,0)=(0,v_3) \nonumber \\
(v_2,v_3)-(0,v_3)=(v_2,0) \nonumber \\
(v_2,v_2)-(v_2,0)=(0,v_2) \nonumber \\
(v_1,v_2)-(0,v_2)=(v_1,0) \nonumber \\
(v_1,v_1)-(v_1,0)=(0,v_1) \nonumber
\end{eqnarray}
and thus $K=span((v_1,0),(0,v_1),(v_2,0),(0,v_2),(v_3,0),(0,v_3) ,(v_4,v_4))$.  The Gram matrix with respect to this basis is
\begin{equation*}
G_K=
\begin{pmatrix}
2 & 0 & -1 & 0 & 0 & 0 & 0 \\
0 & 2 & 0 & -1 & 0  & 0 & 0\\
-1 & 0 & 2 & 0 & -1 & 0 & 0 \\
0 & -1 & 0 & 2 & 0 & -1 & 0\\
0 & 0 & -1 & 0 & 2 & 0 & -1 \\
0 & 0 & 0 & -1 & 0 & 2 & -1 \\
0 & 0 & 0 & 0 & -1 & -1 & 4
\end{pmatrix},
\end{equation*}
which has determinant 40 and Smith invariants 1, 1, 1, 1, 1, 1, 40.  Identifying $span((v_1,-v_1),\\(v_2,-v_2),(v_3,-v_3))$ with $\sqrt{2}A_3$, $K$ contains a sublattice $J\cong\sqrt{2}A_4 \perp \sqrt{2}A_3$ with index 8, and $K=J+\mathbb{Z}(v_1,0)+\mathbb{Z}(v_2,0)+\mathbb{Z}(v_3,0)$.

\subsection{5-Cycle}
There are 24 elements of $O(A_4)$ that correspond to 5-cycles.  We take $h=(e_1,e_2,e_3,e_4,e_5)$, and we have $hv_1=v_2$, $hv_2=v_3$,  $hv_3=v_4$, and $hv_4=-v_1-v_2-v_3-v_4$.  Then $K$ is spanned by $\{ (v_1,v_1),(v_2,v_2),(v_3,v_3),(v_4,v_4),(v_1,v_2),(v_2,v_3),(v_3,v_4),(v_4,-v_1-v_2-v_3-v_4)\}$.  The Gram matrix for this basis is
\begin{equation*}
G_K=
\begin{pmatrix}
4 & -2 & 0 & 0 & 1 & -1 & 0 & -1 \\
-2 & 4 & -2 & 0 & 1  & 1 & -1 & 0\\
0 & -2 & 4 & -2 & -1 & 1 & 1 & -1 \\
0 & 0 & -2 & 4 & 0 & -1 & 1 & 1\\
1 & 1 & -1 & 0 & 4 & -2 & 0 & 0 \\
-1 & 1 & 1 & -1 & -2 & 4 & -2 & 0 \\
0 & -1 & 1 & 1 & 0 & -2 & 4 & -2 \\
-1 & 0 & -1 & 1 & 0 & 0 & -2 & 4 \\
\end{pmatrix},
\end{equation*}
which has determinant 625 and Smith invariants 1, 1, 1, 1, 5, 5, 5, 5.  Thus, $K$ is a sublattice of index 5 in the lattice $L \perp L$.

\subsection{Negative of Transposition}\label{A4NegTrans}
There are 10 elements of $O(A_4)$ that correspond to the negative of a transposition of two unit vectors $e_i,e_j$.  We take $h=-r_{v_1}$.  We have $hv_1=v_1$, $hv_2=-v_1-v_2$,  $hv_3=-v_3$, and $hv_4=-v_4$.  Then $K$ is spanned by $\{ (v_1,v_1),(v_2,v_2),(v_3,v_3),(v_4,v_4),(v_2,-v_1-v_2),\\(v_3,-v_3),(v_4-v_4)\}$. The Gram matrix with respect to this basis is
\begin{equation*}
G_K=
\begin{pmatrix}
4 & -2 & 0 & 0 & -2 & 0 & 0 \\
-2 & 4 & -2 & 0 & 1  & 0 & 0\\
0 & -2 & 4 & -2 & 0 & 0 & 0 \\
0 & 0 & -2 & 4 & 0 & 0 & 0\\
-2 & 1 & 0 & 0 & 4 & -2 & 0 \\
0 & 0 & 0 & 0 & -2 & 4 & -2 \\
0 & 0 & 0 & 0 & 0 & -2 & 4
\end{pmatrix},
\end{equation*} 
which has determinant 1600 and Smith invariants 1, 1, 2, 2, 2, 10, 20.  $K$ contains a sublattice $J\cong\sqrt{2}A_4 \perp \sqrt{2}A_2 \perp \mathbb{Z}w$, where we identify $span((v_3,-v_3),(v_4,-v_4))$ with $\sqrt{2}A_2$ and
\begin{align}
w&=3(v_1,v_1)+6(v_2,-v_1-v_2)+4(v_3,-v_3)+2(v_4-v_4) \nonumber \\
&=(3v_1+6v_2+4v_3+2v_4,-3v_1-6v_2-4v_3-2v_4) \nonumber
\end{align}
is a minimal vector in $ann_K(\sqrt{2}A_4 \perp \sqrt{2}A_2)$ and has norm 60.  Hence, $J$ has determinant $57600=1600\cdot 36$ and thus has index 6 in $K=J+\mathbb{Z}(v_2,-v_1-v_2)$.

\subsection{Negative of Product of Disjoint Transpositions}
There are 15 elements of $O(A_4)$ that correspond to a negative of a product of disjoint transpositions.  We take $h=-r_{v_1}r_{v_3}$, and we have $hv_1=v_1$, $hv_2=-v_1-v_2-v_3$,  $hv_3=v_3$, and $hv_4=-v_3-v_4$.  Then $K$ is spanned by $\{ (v_1,v_1),(v_2,v_2),(v_3,v_3),(v_4,v_4),\\(v_2,-v_1-v_2-v_3),(v_4,-v_3-v_4)\}$.  The Gram matrix with respect to this basis is
\begin{equation*}
G_K=
\begin{pmatrix}
4 & -2 & 0 & 0 & -2 & 0 \\
-2 & 4 & -2 & 0 & 2  & 1\\
0 & -2 & 4 & -2 & -2 & -2 \\
0 & 0 & -2 & 4 & 1 & 1 \\
-2 & 2 & -2 & 1 & 4 & 0 \\
0 & 1 & -2 & 1 & 0 & 4
\end{pmatrix},
\end{equation*}
which has determinant 400 and Smith invariants 1, 1, 1, 1, 20, 20.

\subsection{Negative of 3-Cycle}
There are 20 elements of $O(A_4)$ that correspond to negatives of 3-cycles.  We take $h=-(e_1,e_2,e_3)=-r_{v_1}r_{v_2}$, and we have $hv_1=-v_2$, $hv_2=v_1+v_2$,  $hv_3=-v_1-v_2-v_3$, and $hv_4=-v_4$.  Then $K$ is spanned by $\{ (v_1,v_1),(v_2,v_2),(v_3,v_3),(v_4,v_4),(v_1,-v_2),(v_2,v_1+v_2),\\(v_3,-v_1-v_2-v_3), (v_4, -v_4)\}$.  We may write
\begin{eqnarray}
(v_2,v_1+v_2)-(v_2,v_2)=(0,v_1) \nonumber \\
(v_1,v_1)-(0,v_1)=(v_1,0) \nonumber \\
(v_1,0)-(v_1,-v_2)=(0,v_2) \nonumber \\
(v_2,v_2)-(0,v_2)=(v_2,0) \nonumber
\end{eqnarray}
and thus $K$ is spanned by $\{(v_1,0),(0,v_1),(v_2,0),(0,v_2),(v_3,v_3),(v_4,v_4),(v_3,-v_3),(v_4, -v_4)\}$.  This gives the Gram matrix
\begin{equation*}
G_K=
\begin{pmatrix}
2 & 0 & -1 & 0 & 0 & 0 & 0 & 0\\
0 & 2 & 0 & -1 & 0  & 0 & 0 & 0\\
-1 & 0 & 2 & 0 & -1 & 0 & -1 & 0 \\
0 & -1 & 0 & 2 & -1 & 0 & 1 & 0\\
0 & 0 & -1 & -1 & 4 & -2 & 0  & 0\\
0 & 0 & 0 & 0 & -2 & 4 & 0 & 0 \\
0 & 0 & -1 & -1 & 0 & 0 & 4 & -2 \\
0 & 0 & 0 & 0 & 0 & 0 & -2 & 4
\end{pmatrix},
\end{equation*} 
which has determinant 400 and Smith invariants 1, 1, 1, 1, 2, 2, 10, 10.  Identifying $span((v_1,-v_1),(v_2,-v_2),(v_3,-v_3),(v_4,-v_4))$ with $\sqrt{2}A_4$, we see that $K$ contains a sublattice $J\cong\sqrt{2}A_4 \perp \sqrt{2}A_4$ with index 4, and we may write $K=J+\mathbb{Z}(v_1,0)+\mathbb{Z}(v_2,0)$.

\subsection{Negative of Product of 3-Cycle and Transposition}
There are 20 elements of $O(A_4)$ that correspond to negatives of products of disjoint 3-cycles and transpositions.  We take $h=-(e_1,e_2,e_3)(e_4,e_5)=-r_{v_1}r_{v_2}r_{v_4}$, and we have $hv_1=-v_2$, $hv_2=v_1+v_2$,  $hv_3=-v_1-v_2-v_3-v_4$, and $hv_4=v_4$.  Then $K$ is spanned by $\{ (v_1,v_1),(v_2,v_2),(v_3,v_3),(v_4,v_4),(v_1,-v_2),(v_2,v_1+v_2),(v_3,-v_1-v_2-v_3-v_4)\}$. We have
\begin{eqnarray}
(v_2,v_1+v_2)-(v_2,v_2)=(0,v_1) \nonumber \\
(v_1,v_1)-(0,v_1)=(v_1,0) \nonumber \\
(v_1,0)-(v_1,-v_2)=(0,v_2) \nonumber \\
(v_2,v_2)-(0,v_2)=(v_2,0) \nonumber
\end{eqnarray}
and thus $K=span((v_1,0),(0,v_1),(v_2,0),(0,v_2),(v_3,v_3),(v_4,v_4),(v_3,-v_3-v_4))$.  The Gram matrix
\begin{equation*}
G_K=
\begin{pmatrix}
2 & 0 & -1 & 0 & 0 & 0 & 0 \\
0 & 2 & 0 & -1 & 0  & 0 & 0\\
-1 & 0 & 2 & 0 & -1 & 0 & -1 \\
0 & -1 & 0 & 2 & -1 & 0 & 1\\
0 & 0 & -1 & -1 & 4 & -2 & 1 \\
0 & 0 & 0 & 0 & -2 & 4 & -2 \\
0 & 0 & -1 & 1 & 1 & -2 & 4
\end{pmatrix}
\end{equation*} 
has determinant 100 and Smith invariants 1, 1, 1, 1, 5, 20. $K$ contains a sublattice $J\cong\sqrt{2}A_4 \perp \sqrt{2}A_2 \perp \mathbb{Z}w$, where we identify $span((v_1,-v_1),(v_2,-v_2))$ with $\sqrt{2}A_2$ and
\begin{align}
w&=2(v_1,0)-2(0,v_1)+4(v_2,0)-4(0,v_2)+3(v_4,v_4)+6(v_3,-v_3-v_4) \nonumber \\
&=(2v_1+4v_2+6v_3+3v_4,-2v_1-4v_2-6v_3-3v_4) \nonumber
\end{align}
is a minimal vector in $ann_K(\sqrt{2}A_4 \perp \sqrt{2}A_2)$ and has norm 60.  Hence, $J$ has determinant $57600=100\cdot 24^2$ and thus has index 24 in $K=J+\mathbb{Z}(v_1,0)+\mathbb{Z}(v_2,0)+\mathbb{Z}(v_3,-v_3-v_4)$.  The $J$ identified here is isometric to the $J$ from \ref{A4NegTrans}, and the lattice $K$ in the former case is a sublattice of index 4 in the current lattice.

\subsection{Negative of 4-Cycle}
There are 30 elements of $O(A_4)$ that correspond to negatives of 4-cycles.  We take $h=-(e_1,e_2,e_3,e_4)$, and we have $hv_1=-v_2$, $hv_2=-v_3$,  $hv_3=v_1+v_2+v_3$, and $hv_4=\\-v_1-v_2-v_3-v_4$.  Then $K$ is spanned by $\{ (v_1,v_1),(v_2,v_2),(v_3,v_3),(v_4,v_4),(v_1,-v_2),(v_2,-v_3),\\(v_3,v_1+v_2+v_3),(v_4,-v_1-v_2-v_3-v_4)\}$.  We have
\begin{equation*}
(v_3,v_1+v_2+v_3)=(v_3,v_3)+(v_1,v_1)-(v_1,-v_2)
\end{equation*}
and the remaining vectors are linearly independent, so $K=span((v_1,v_1),(v_2,v_2),(v_3,v_3),(v_4,v_4),\\
(v_1,-v_2),(v_2,-v_3),(v_4,-v_1-v_2-v_3-v_4))$ has Gram matrix
\begin{equation*}
G_K=
\begin{pmatrix}
4 & -2 & 0 & 0 & 3 & -1 & -1 \\
-2 & 4 & -2 & 0 & -3  & 3 & 0\\
0 & -2 & 4 & -2 & 1 & -3 & -1 \\
0 & 0 & -2 & 4 & 0 & 1 & 1\\
3 & -3 & 1 & 0 & 4 & -2 & 0 \\
-1 & 3 & -3 & 1 & -2 & 4 & 0 \\
-1 & 0 & -1 & 1 & 0 & 0 & 4
\end{pmatrix},
\end{equation*}
determinant 200, and Smith invariants 1, 1, 1, 1, 5, 40.

\subsection{Negative of 5-Cycle}
There are 24 elements of $O(A_4)$ that correspond to negatives of 5-cycles.  We take $h=-(e_1,e_2,e_3,e_4,e_5)$, and we have $hv_1=-v_2$, $hv_2=-v_3$,  $hv_3=-v_4$, and $hv_4=v_1+v_2+v_3+v_4$.  Then $K$ is spanned by $\{ (v_1,v_1),(v_2,v_2),(v_3,v_3),(v_4,v_4),(v_1,-v_2),(v_2,-v_3),(v_3,-v_4),\\(v_4,v_1+v_2+v_3+v_4)\}$.  We have
\begin{align}
(v_4,v_1+v_2+v_3+v_4)-(v_4,v_4)+(v_2,-v_3)-(v_2,v_2)&=(0,v_1) \nonumber \\
(v_1,v_1)-(0,v_1)&=(v_1,0) \nonumber \\
(v_1,0)-(v_1,-v_2)&=(0,v_2) \nonumber \\
(v_2,v_2)-(0,v_2)&=(v_2,0) \nonumber \\
(v_2,0)-(v_2,-v_3)&=(0,v_3) \nonumber \\
(v_3,v_3)-(0,v_3)&=(v_3,0) \nonumber \\
(v_3,0)-(v_3,-v_4)&=(0,v_4) \nonumber \\
(v_4,v_4)- (0, v_4)&=(v_4,0) \nonumber
\end{align}
and these vectors span the whole lattice $L \perp L$.

\section{$L=D_4$}
The isometry group of $D_4$ has order $1152=2^7 \cdot 3^2=3 \cdot 2^4 \cdot 4!$.  It contains the monomial group on the orthonormal basis ${e_1,e_2,e_3,e_4}\subset \mathbb{R}^4$ with index 3, and acts as $Sym_3$ on both the set of frames of roots in $L$ and the set of frames of unit vectors in $L^*$.\footnote{By \emph{frame}, we mean a set of 8 vectors in which any two vectors are either orthogonal or negatives of each other.  See \cite{Griess} for more details on the frames and the isometry group of $D_4$.}  There are 25 conjugacy classes in $O(D_4)$.  This fact along with other useful information about the group is taken from \cite{Atlas}. We take the standard basis $\{v_1=e_1-e_2,v_2=e_2-e_3,v_3=e_3-e_4,v_4=e_3+e_4\}$  for $D_4$.

\subsection{Reflection at a root ($\pm e_i\pm e_j$)}
Let $h$ be the reflection $r_{v_1}$, which switches $e_1$ and $e_2$ and fixes $e_3$ and $e_4$.  We have $hv_1=-v_1$, $hv_2=v_1+v_2$,  $hv_3=v_3$, and $hv_4=v_4$.  Then $K$ is spanned by $\{ (v_1,v_1),(v_2,v_2),(v_3,v_3),\\(v_4,v_4),(v_1,-v_1),(v_2,v_1+v_2)\}$.  Since $(v_2,v_1+v_2)-(v_2,v_2)=(0,v_1)$ and $(v_1,v_1)-(0,v_1)=(v_1,0)$, $K=span((v_1,0),(0,v_1),(v_2,v_2),(v_3,v_3),(v_4,v_4))$.  This gives the Gram matrix
\begin{equation*}G_K=
\begin {pmatrix} 2&0&-1&0&0\\
0&2&-1&0&0\\
-1&-1&4&-2&-2\\
0&0&-2&4&0\\
0&0&-2&0&4\end {pmatrix},
\end{equation*}
which has determinant 64 and Smith invariants 1, 1, 4, 4, 4.  $K$ contains a sublattice $J\cong\sqrt{2}D_4 \perp \mathbb{Z}(v_1,-v_1)$ with index 2, and $K=J+\mathbb{Z}(v_1,0)$.  Since every vector in $K^*$ has order dividing 4, we have $(2v,2w)=(4v,w)\in \mathbb{Z}$ for all $v,w\in K^*$, so $H:=K+2K^*$ is an integral lattice, and $H$ contains $K$ with index 8.  Thus, $H$ has determinant 1, so $H\cong \mathbb{Z}^5$ by Theorem \ref{Unimodular}.

\subsection{Negation of One Unit Vector}
Let $h$ be the element of $O(D_4)$ which negates $e_1$ and fixes $e_2$, $e_3$ and $e_4$.  By Theorem \ref{DnNegations}, we have $K\cong \sqrt{2}D_5$.

\subsection{Negation of Two Unit Vectors/Product of Reflections at Orthogonal Roots}
Let $h$ be the element of $O(D_4)$ which negates $e_1$ and $e_2$ and fixes $e_3$ and $e_4$.  We can write $h$ as a product of reflections at orthogonal roots: $h=r_{e_1+e_2}r_{e_1-e_2}$.  By Theorem \ref{DnNegations}, $K \cong \sqrt{2}D_6$.

\subsection{Negation of Three Unit Vectors}
Let $h$ be the element of $O(D_4)$ which negates $e_1$, $e_2$, and $e_3$ and fixes $e_4$.  By Theorem \ref{DnNegations}, $K\cong \sqrt{2}D_7$.

\subsection{Product of Reflection at a Root and Negation of the Remaining Two Unit Vectors}
Let $h$ be the element of $O(D_4)$ which switches $e_1$ and $e_2$ and negates $e_3$ and $e_4$.  We can write $h=r_{e_1-e_2}r_{e_3-e_4}r_{e_3+e_4}$.  We have $hv_1=-v_1$, $hv_2=v_1+v_2+v_3+v_4$,  $hv_3=-v_3$, and $hv_4=-v_4$.  Then $K$ is spanned by $\{ (v_1,v_1),(v_2,v_2),(v_3,v_3),(v_4,v_4),(v_1,-v_1),(v_2,v_1+v_2+v_3+v_4),(v_3,-v_3),(v_4,-v_4)\}$.  We can write the last vector as a linear combination of the others:
\begin{equation*}
(v_4,-v_4)=(v_1,v_1)+2(v_2,v_2)+(v_3,v_3)+(v_4,v_4)-(v_1,-v_1)-2(v_2,v_1+v_2+v_3+v_4)-(v_3,-v_3).
\end{equation*}
The remaining vectors are linearly independent and give the Gram matrix
\begin{equation*} G_K=
\begin {pmatrix}
4&-2&0&0&0&0&0\\
-2&4&-2&-2&0&1&0\\
0&-2&4&0&0&0&0\\
0&-2&0&4&0&0&0\\
0&0&0&0&4&-2&0\\
0&1&0&0&-2&4&-2\\
0&0&0&0&0&-2&4
\end{pmatrix},
\end{equation*}
which has determinant 1024 and Smith invariants 1, 1, 4, 4, 4, 4, 4.  Since every vector in $K^*$ has order dividing 4, $H:=K+2K^*$ is an integral lattice.  Since $H$ contains $K$ with index 32, $H$ has determinant 1 and is therefore isometric to $\mathbb{Z}^7$.

\subsection{Product of Reflection at a Root and Negation of One of the Remaining Unit Vectors}
Let $h$ be the element of $O(D_4)$ which switches $e_1$ and $e_2$, negates $e_3$, and fixes $e_4$.  We have $hv_1=-v_1$, $hv_2=v_1+v_2+v_3+v_4$,  $hv_3=-v_4$, and $hv_4=-v_3$.  Then $K$ is spanned by $\{ (v_1,v_1),(v_2,v_2),(v_3,v_3),(v_4,v_4),(v_1,-v_1),(v_2,v_1+v_2+v_3+v_4),(v_3,-v_4),(v_4,-v_3)\}$.  We may write
\begin{eqnarray}
(v_1,v_1)+2(v_2,v_2)+2(v_3,v_3)-2(v_2,v_1+v_2+v_3+v_4)-2(v_3,-v_4)=(v_1,-v_1)\nonumber \\
-(v_3,v_3)+(v_4,v_4)+(v_3,-v_4)=(v_4,-v_3),\nonumber
\end{eqnarray}
and the remaining six vectors are linearly independent and give the Gram matrix
\begin{equation*} G_K=
\begin {pmatrix}
4&-2&0&0&0&0\\
-2&4&-2&-2&1&0\\
0&-2&4&0&0&2\\
0&-2&0&4&0&-2\\
0&1&0&0&4&-2\\
0&0&2&-2&-2&4
\end{pmatrix},
\end{equation*}
which has determinant 128 and Smith invariants 1, 1, 2, 2, 4, 8.  If we let $u_1$ and $u_2$ be two independent vectors in $K^*$ with order $4 \pmod{K}$, we obtain the integral lattice $J:=K+2\mathbb{Z}u_1+2\mathbb{Z}u_2$, which contains $K$ with index 4 and thus has determinant 8.  If the discriminant group of $J$ is not elementary abelian, there exists a vector $u_3$ with order $4 \pmod{J}$ and we can embed $J$ in the integral lattice $J+2\mathbb{Z}u_3$, which has determinant 2.  Otherwise, there exists a nontrivial coset $J+u_3$ such that $u_3$ has integer norm (Theorem \ref{2group}), and hence $J+\mathbb{Z}u_3$ is an integral lattice of determinant 2.  In either case, this lattice is either rectangular or isometric to $E_7$ (Theorem \ref{E7uniqueness}).

\subsection{3-Cycle of Unit Vectors}
Let $h$ be the element of $O(D_4)$ which acts as the 3-cycle $(e_1,e_2,e_3)$ on the standard basis of $\mathbb{R}^4$.  We can write $h=r_{e_1-e_2}r_{e_2-e_3}$.  We have $hv_1=v_2$, $hv_2=-v_1-v_2$,  $hv_3=v_1+v_2+v_3$, and $hv_4=v_1+v_2+v_4$.  Then $K$ is spanned by $\{ (v_1,v_1),(v_2,v_2),(v_3,v_3),(v_4,v_4),(v_1,v_2),(v_2,-v_1-v_2),(v_3,v_1+v_2+v_3),(v_4,v_1+v_2+v_4)\}$.
We may write
\begin{eqnarray}
-(v_3,v_3)+(v_4,v_4)+(v_3,v_1+v_2+v_3)=(v_4,v_1+v_2+v_4) \nonumber \\
(v_1,v_1)+2(v_2,v_2)+3(v_3,v_3)-2(v_2,-v_1-v_2)-3(v_3,v_1+v_2+v_3)=(v_1,v_2), \nonumber
\end{eqnarray}
and the remaining six vectors give the Gram matrix
\begin{equation*} G_K=
\begin {pmatrix}
4&-2&0&0&-2&1\\
-2&4&-2&-2&1&-1\\
0&-2&4&0&0&3\\
0&-2&0&4&0&-1\\
-2&1&0&0&4&-2\\
1&-1&3&-1&-2&4
\end{pmatrix},
\end{equation*}
which has determinant 48 and Smith invariants 1, 1, 1, 1, 4, 12.  We can then embed $K$ with index 4 in an integral lattice of determinant 3 by choosing vectors of order 4 in the dual. Such a lattice is either rectangular or isometric to $A_2 \perp \mathbb{Z}^4$ or $E_6$ (Theorem \ref{E6uniqueness}).

\subsection{3-cycle of Roots} \label{D4Root3Cycle}
Let $h$ be the element of $O(D_4)$ which acts as the 3-cycle $(v_1,v_3,v_4)$ on the vectors in the frame (of roots) containing these roots (and fixes the remaining vector in this frame, which is $e_1+e_2=v_1+2v_2+v_3+v_4$).   We can express $h$ in this case as the matrix
\begin{equation*}
\begin{pmatrix}
1/2&1/2&1/2&1/2\\
1/2&1/2&-1/2&-1/2\\
1/2&-1/2&1/2&-1/2\\
-1/2&1/2&1/2&-1/2
\end{pmatrix}.
\end{equation*}
In this case, we have $hv_1=v_3$, $hv_2=v_2$,  $hv_3=v_4$, and $hv_4=v_1$.  Then $K$ is spanned by $\{ (v_1,v_1),(v_2,v_2),(v_3,v_3),(v_4,v_4),(v_1,v_3),(v_3,v_4),(v_4,v_1)\}$.  We can write $(v_4,v_1)=(v_1,v_1)+(v_3,v_3)+(v_4,v_4)-(v_1,v_3)-(v_3,v_4)$, and the remaining vectors are linearly independent and give the Gram matrix
\begin{equation*} G_K=
\begin {pmatrix}
4&-2&0&0&2&0\\
-2&4&-2&-2&-2&-2\\
0&-2&4&0&2&2\\
0&-2&0&4&0&2\\
2&-2&2&0&4&0\\
0&-2&2&2&0&4
\end{pmatrix},
\end{equation*}
which has determinant 192 and Smith invariants 2, 2, 2, 2, 2, 6.  We see that $K/\sqrt{2}$ is an even integral lattice of determinant 3, and is thus isometric to $E_6$ (Theorem \ref{E6uniqueness}), so $K \cong \sqrt{2}E_6$.

\subsection{Element of Order 3 That Fixes No Frames of Roots or Unit Vectors} \label{D4-3}
The 3-cycles mentioned above all fix either a frame of roots or a frame of unit vectors; we now seek an element which fixes no frames of either type.  Let $h$ be the element represented by the matrix
\begin{equation*}
\begin{pmatrix}
-1/2&-1/2&-1/2&-1/2\\
1/2&-1/2&1/2&-1/2\\
1/2&-1/2&-1/2&1/2\\
1/2&1/2&-1/2&-1/2
\end{pmatrix}.
\end{equation*}
We can check that this matrix satisifies the condition $h^2+h+1=0$, so by Theorem \ref{A2Tensor} $K\cong A_2\otimes D_4$.

\subsection{4-Cycle of Unit Vectors}
Let $h$ be the element of $O(D_4)$ which acts as the 4-cycle $(e_1,e_2,e_3,e_4)$ on the standard basis of $\mathbb{R}^4$.  We can write $h=r_{e_1-e_2}r_{e_2-e_3}r_{e_3-e_4}$.  We have $hv_1=v_2$, $hv_2=v_3$,  $hv_3=-v_1-v_2-v_3$, and $hv_4=v_1+v_2+v_4$.  Then $K$ is spanned by $\{ (v_1,v_1),(v_2,v_2),(v_3,v_3),(v_4,v_4),(v_1,v_2),(v_2,v_3),(v_3,-v_1-v_2-v_3),(v_4,v_1+v_2+v_4)\}$.  We may write $(v_3,-v_1-v_2-v_3)=(v_1,v_1)+2(v_2,v_2)+(v_3,v_3)+2(v_4,v_4)-(v_1,v_2)-2(v_2,v_3)-2(v_4,v_1+v_2+v_4)$, and the remaning vectors give the Gram matrix
\begin{equation*} G_K=
\begin {pmatrix}
4&-2&0&0&1&-1&1\\
-2&4&-2&-2&1&1&-1\\
0&-2&4&0&-1&1&-1\\
0&-2&0&4&-1&-1&3\\
1&1&-1&-1&4&-2&0\\
-1&1&1&-1&-2&4&-2\\
1&-1&-1&3&0&-2&4
\end{pmatrix},
\end{equation*}
which has determinant 128 and Smith invariants 1, 1, 2, 2, 2, 2, 8.  By choosing a vector $u \in K^*$ of order $4 \pmod{K}$, we can embed $K$ with index 2 in the integral lattice $K+2\mathbb{Z}u$ of determinant 32.  If there exists a vector $w$ of order 4 in the discriminant group of this lattice, we can further embed $K$ in a lattice of determinant 8 in the same way.  Otherwise, the discriminant group is elementary abelian of order 32 and we can choose vector $w$ of order 2 that has integer norm (Theorem \ref{2group}).  Then $K+2\mathbb{Z}u+\mathbb{Z}w$ is an integral lattice of determinant 8.  Using the same method again, we can finally embed $K$ in an integral lattice of determinant 2, which is either rectangular or isometric to $E_7$ (Theorem \ref{E7uniqueness}).

We can also embed $K$ in the $A_7$ lattice, since by Theorem \ref{ncyclethm}, the overlattice obtained by applying $h$ to $\mathbb{Z}^4$ is isometric to $A_7$.

\subsection{Transposition of Unit Vectors with a Negation}
Let $h$ be the element of $O(D_4)$ which sends $e_1$ to $-e_2$, $e_2$ to $e_1$, and fixes $e_3$ and $e_4$.  We have $hv_1=-v_1-2v_2-v_3-v_4$, $hv_2=v_1+v_2$,  $hv_3=v_3$, and $hv_4=v_4$, so $K$ is spanned by $\{ (v_1,v_1),(v_2,v_2),(v_3,v_3),(v_4,v_4),(v_1,-v_1-2v_2-v_3-v_4),(v_2,v_1-v_2)\}$.  This gives the Gram matrix
\begin{equation*} G_K=
\begin {pmatrix}
4&-2&0&0&2&0\\
-2&4&-2&-2&-2&3\\
0&-2&4&0&0&-2\\
0&-2&0&4&0&-2\\
2&-2&0&0&4&-2\\
0&3&-2&-2&-2&4
\end{pmatrix},
\end{equation*}
which has determinant 64 and Smith invariants 1, 1, 2, 2, 4, 4.  By choosing two independent vectors $u_1$ and $u_2$ with order 4 in the discriminant group, we can embed $K$ with index 4 in the integral lattice $J:=K+2\mathbb{Z}u_1+2\mathbb{Z}u_2$ of determinant 4.  If the discriminant group of $J$ is not elementary abelian, there exists a vector $u_3$ with order $4 \pmod{J}$ and we can embed $J$ in the integral lattice $J+2\mathbb{Z}u_3$.  Otherwise, there exists a nontrivial coset $J+u_3$ such that $u_3$ has integer norm (Theorem \ref{2group}), and hence $J+\mathbb{Z}u_3$ is an integral lattice.  In each case, the resulting integral lattice contains $J$ with index 2, and thus has determinant 1 and is isometric to $\mathbb{Z}^6$.

\subsection{Product of Transposition of Unit Vectors with a Negation and Negation of the Other Unit Vectors}
Let $h$ be the element of $O(D_4)$ which sends $e_1$ to $-e_2$, $e_2$ to $e_1$ and negates $e_3$ and $e_4$.  We have $hv_1=-v_1-2v_2-v_3-v_4$, $hv_2=v_1+v_2+v_3+v_4$,  $hv_3=-v_3$, and $hv_4=-v_4$, so $K$ is spanned by $\{ (v_1,v_1),(v_2,v_2),(v_3,v_3),(v_4,v_4),(v_1,-v_1-2v_2-v_3-v_4),(v_2,v_1+v_2+v_3+v_4),(v_3,-v_3),(v_4,-v_4)\}$.  These vectors give the Gram matrix
\begin{equation*} G_K=
\begin {pmatrix}
4&-2&0&0&2&0&0&0\\
-2&4&-2&-2&-2&1&0&0\\
0&-2&4&0&0&0&0&0\\
0&-2&0&4&0&0&0&0\\
2&-2&0&0&4&-2&0&0\\
0&1&0&0&-2&4&-2&-2\\
0&0&0&0&0&-2&4&0\\
0&0&0&0&0&-2&0&4
\end{pmatrix},
\end{equation*}
which has determinant 1024 and Smith invariants 1, 1, 2, 2, 4, 4, 4, 4.  Thus, $K$ is a sublattice of index 8 in $D_4 \perp D_4$.

\subsection{Product of Transposition of Unit Vectors with a Negation and Transposition of the Remaining Unit Vectors}
Let $h$ be the element of $O(D_4)$ which sends $e_1$ to $-e_2$, $e_2$ to $e_1$ and switches $e_3$ and $e_4$.   We have $hv_1=-v_1-2v_2-v_3-v_4$, $hv_2=v_1+v_2+v_3$,  $hv_3=-v_3$, and $hv_4=v_4$, so $K$ is spanned by $\{ (v_1,v_1),(v_2,v_2),(v_3,v_3),(v_4,v_4),(v_1,-v_1-2v_2-v_3-v_4),(v_2,v_1+v_2+v_3),(v_3,-v_3)\}$.  These vectors give the Gram matrix
\begin{equation*} G_K=
\begin {pmatrix}
4&-2&0&0&2&0&0\\
-2&4&-2&-2&-2&2&0\\
0&-2&4&0&0&0&0\\
0&-2&0&4&0&-2&0\\
2&-2&0&0&4&-2&0\\
0&2&0&-2&-2&4&-2\\
0&0&0&0&0&-2&4
\end{pmatrix},
\end{equation*}
which has determinant 256 and Smith invariants 2, 2, 2, 2, 2, 2, 4.  The lattice $K/\sqrt{2}$ is therefore an even integral lattice of determinant 2, which is isometric to $E_7$ (Theorem \ref{E7uniqueness}), so $K\cong \sqrt{2}E_7$.

\subsection{Product of Transposition of Unit Vectors with a Negation and Transposition of the Remaining Unit Vectors with a Negation}
Let $h$ be the element of $O(D_4)$ which sends $e_1$ to $-e_2$, $e_2$ to $e_1$, $e_3$ to $-e_4$, and $e_4$ to $e_3$.  We have $hv_1=-v_1-2v_2-v_3-v_4$, $hv_2=v_1+v_2+v_4$,  $hv_3=-v_4$, and $hv_4=v_3$, so $K$ is spanned by $\{ (v_1,v_1),(v_2,v_2),(v_3,v_3),(v_4,v_4),(v_1,-v_1-2v_2-v_3-v_4),(v_2,v_1+v_2+v_4),(v_3,-v_4),(v_4,v_3)\}$.  These vectors give the Gram matrix
\begin{equation*} G_K=
\begin {pmatrix}
4&-2&0&0&2&0&0&0\\
-2&4&-2&-2&-2&2&0&-2\\
0&-2&4&0&0&-2&2&2\\
0&-2&0&4&0&0&-2&2\\
2&-2&0&0&4&-2&0&0\\
0&2&-2&0&-2&4&-2&-2\\
0&0&2&-2&0&-2&4&0\\
0&-2&2&2&0&-2&0&4
\end{pmatrix},
\end{equation*}
which has determinant 256 and Smith invariants 2, 2, 2, 2, 2, 2, 2, 2.  The lattice $K/\sqrt{2}$ is therefore an even integral lattice of determinant 1, which is isometric to $E_8$ and thus $K\cong \sqrt{2}E_8$.

\subsection{Negative of 3-Cycle of Unit Vectors}
Let $h$ be the element of $O(D_4)$ which acts as the negative of the 3-cycle $(e_1,e_2,e_3)$ on the standard basis of $\mathbb{R}^4$.  We have $hv_1=-v_2$, $hv_2=v_1+v_2$,  $hv_3=-v_1-v_2-v_3$, and $hv_4=-v_1-v_2-v_4$.  Then $K$ is spanned by $\{ (v_1,v_1),(v_2,v_2),(v_3,v_3),(v_4,v_4),(v_1,-v_2),(v_2,v_1+v_2),(v_3,-v_1-v_2-v_3),(v_4,-v_1-v_2-v_4)\}$.  This gives the Gram matrix
\begin{equation*} G_K=
\begin {pmatrix}
4&-2&0&0&3&0&-1&-1\\
-2&4&-2&-2&-3&3&-1&-1\\
0&-2&4&0&1&-2&1&1\\
0&-2&0&4&1&-2&1&1\\
3&-3&1&1&4&-2&0&0\\
0&3&-2&-2&-2&4&-2&-2\\
-1&-1&1&1&0&-2&4&0\\
-1&-1&1&1&0&-2&0&4
\end{pmatrix},
\end{equation*}
which has determinant 256 and Smith invariants 1, 1, 1, 1, 4, 4, 4, 4, and thus $K$ is a sublattice of index 4 in $D_4\perp D_4$.

\subsection{Negative of 3-cycle of Roots} 
Let $h$ be the negative of the element discussed in \ref{D4Root3Cycle}, which acts as the 3-cycle $(v_1,v_3,v_4)$ on the vectors in the frame (of roots) containing these roots.  In this case, we have $hv_1=-v_3$, $hv_2=-v_2$,  $hv_3=-v_4$, and $hv_4=-v_1$.  Then $K$ is spanned by $\{ (v_1,v_1),(v_2,v_2),(v_3,v_3),\\(v_4,v_4),(v_1,-v_3),(v_2,-v_2),(v_3,-v_4),(v_4,-v_1)\}$.  These vectors give the Gram matrix
\begin{equation*} G_K=
\begin {pmatrix}
4&-2&0&0&2&0&0&-2\\
-2&4&-2&-2&0&0&0&0\\
0&-2&4&0&-2&0&2&0\\
0&-2&0&4&0&0&-2&2\\
2&0&-2&0&4&-2&0&0\\
0&0&0&0&-2&4&-2&-2\\
0&0&2&-2&0&-2&4&0\\
-2&0&0&2&0&-2&0&4
\end{pmatrix},
\end{equation*}
which has determinant 256 and Smith invariants 2, 2, 2, 2, 2, 2, 2, 2.  The lattice $K/\sqrt{2}$ is therefore an even integral lattice of determinant 1, which is isometric to $E_8$ and thus $K\cong \sqrt{2}E_8$.

\subsection{Negative of Element of Order 3 That Fixes No Frames of Roots or Unit Vectors}
Let $h$ be the negative of the element discussed in \ref{D4-3}.  We have $hv_1=-v_2-v_3-v_4$, $hv_2=v_2+v_3$,  $hv_3=-v_2$, and $hv_4=v_1+v_2+v_4$. Then $K$ is spanned by $\{ (v_1,v_1),(v_2,v_2),(v_3,v_3),\\(v_4,v_4),(v_1,-v_2-v_3-v_4),(v_2,v_2+v_3),(v_3,-v_2),(v_4,v_1+v_2+v_4)\}$.  The Gram matrix
\begin{equation*} G_K=
\begin {pmatrix}
4&-2&0&0&3&-2&1&1\\
-2&4&-2&-2&-1&3&-3&-1\\
0&-2&4&0&-1&0&3&-1\\
0&-2&0&4&-1&-2&1&3\\
3&-1&-1&-1&4&-2&0&0\\
-2&3&0&-2&-2&4&-2&-2\\
1&-3&3&1&0&-2&4&0\\
1&-1&-1&3&0&-2&0&4
\end{pmatrix}
\end{equation*}
has determinant 16 and thus $K$ is the whole lattice $D_4 \perp D_4$.

\subsection{Product of 3-cycle of Unit Vectors and Negation of the Remaining Unit Vector}
Let $h$ be the element of $O(D_4)$ which acts as the 3-cycle $(e_1,e_2,e_3)$ and negates $e_4$.  We have $hv_1=v_2$, $hv_2=-v_1-v_2$,  $hv_3=v_1+v_2+v_4$, and $hv_4=v_1+v_2+v_3$, so $K$ is spanned by $\{ (v_1,v_1),(v_2,v_2),(v_3,v_3),(v_4,v_4),(v_1,v_2),(v_2,-v_1-v_2),(v_3,v_1+v_2+v_4),(v_4,v_1+v_2+v_3)\}$.  We may write
\begin{align}
(v_1,v_1)+(v_2,v_2)+(v_3,v_3)+(v_4,v_4)-(v_2,-v_1-v_2)&\nonumber \\
\mbox{}-(v_3,v_1+v_2+v_4)-(v_4,v_1+v_2+v_3)&=(v_1,0)\nonumber \\
(v_1,v_1)-(v_1,0)&=(0,v_1) \nonumber \\
(v_1,v_2)-(v_1,0)&=(0,v_2) \nonumber \\
(v_2,v_2)-(0,v_2)&=(v_2,0) \nonumber \\
(v_3,v_1+v_2+v_4)-(0,v_1)-(0,v_2)&=(v_3,v_4) \nonumber \\
(v_4,v_4)-(v_3,v_4)+(v_3,v_3)+(0,v_1)+(0,v_2)&=(v_4,v_1+v_2+v_3). \nonumber
\end{align}
Thus, $K$ is spanned by $\{(v_1,0),(0,v_1),(v_2,0),(0,v_2),(v_3,v_3),(v_4,v_4),(v_3,v_4)\}$.  These vectors give the Gram matrix
\begin{equation*} G_K=
\begin {pmatrix}
2&0&-1&0&0&0&0\\
0&2&0&-1&0&0&0\\
-1&0&2&0&-1&-1&-1\\
0&-1&0&2&-1&-1&-1\\
0&0&-1&-1&4&0&2\\
0&0&-1&-1&0&4&2\\
0&0&-1&-1&2&2&4
\end{pmatrix},
\end{equation*}
which has determinant 96 and Smith invariants 1, 1, 1, 1, 2, 2, 24.  We can thus embed $K$ with index 2 in an integral lattice of determinant 24.  We can also identify the lattice $M$ with $\sqrt{2}D_4$ and the span of $\{(v_1,-v_1),(v_2,-v_2)\}$ with $\sqrt{2}A_2$, giving a sublattice $J\cong\sqrt{2}D_4 \perp \sqrt{2}A_2 \perp \mathbb{Z}(v_3-v_4,v_4-v_3)$, which has determinant $64\cdot 12 \cdot 8=96\cdot64$ and thus has index 8 in $K$.

\subsection{Negative of 3-cycle of Unit Vectors, fixing the Remaining Unit Vector}
Let $h$ be the negative of the element above, so that $h$ sends $e_1$ to $-e_2$, $e_2$ to $-e_3$, $e_3$ to $-e_1$, and fixes $e_4$. We have $hv_1=-v_2$, $hv_2=v_1+v_2$,  $hv_3=-v_1-v_2-v_4$, and $hv_4=-v_1-v_2-v_3$, and $K$ is spanned by $\{ (v_1,v_1),(v_2,v_2),(v_3,v_3),(v_4,v_4),(v_1,-v_2),(v_2,v_1+v_2),(v_3,-v_1-v_2-v_4),(v_4,-v_1-v_2-v_3)\}$.  We can write $(v_4,-v_1-v_2-v_3)=(v_4,v_4)-(v_3,v_3)+(v_3,-v_1-v_2-v_4)$, and the remaining vectors are linearly independent and give the Gram matrix
\begin{equation*} G_K=
\begin {pmatrix}
4&-2&0&0&3&0&-1\\
-2&4&-2&-2&-3&3&-1\\
0&-2&4&0&1&-2&3\\
0&-2&0&4&1&-2&-1\\
3&-3&1&1&4&-2&0\\
0&3&-2&-2&-2&4&-2\\
-1&-1&3&-1&0&-2&4
\end{pmatrix},
\end{equation*}
which has determinant 32 and Smith invariants 1, 1, 1, 1, 2, 2, 8.  Using a vector of order 4 in the discriminant group, we can embed $K$ in an integral lattice of determinant 8.  Then, either there is yet another vector of order 4, or the discriminant group is elementary abelian and we can find a nontrivial vector of integer norm.  In either case, we can embed $K$ in an integral lattice of determinant 2, which must be either rectangular or isometric to $E_7$.

\subsection{Element of Order 6 Acting as a 3-cycle of Frames of Unit Vectors and Transposition of Frames of Roots}
Let $h$ be the element given by the matrix
\begin{equation*}
\begin {pmatrix}
1/2&-1/2&1/2&1/2\\
1/2&1/2&-1/2&1/2\\
1/2&1/2&1/2&-1/2\\
1/2&-1/2&-1/2&-1/2
\end{pmatrix},
\end{equation*}
which has order 6 and acts as a 3-cycle on the set of frames of unit vectors and a transposition on the set of frames of roots.
We have $hv_1=v_1+v_2+v_4$, $hv_2=-v_1$, $hv_3=-v_2$, and $hv_4=v_1+v_2+v_3$.  Thus, $K$ is spanned by $\{ (v_1,v_1),(v_2,v_2),(v_3,v_3),(v_4,v_4),(v_1,v_1+v_2+v_4),(v_2,-v_1),(v_3,-v_2),(v_4,v_1+v_2+v_3)\}$.  We may write $(v_4,v_1+v_2+v_3)=(v_1,v_1)+(v_2,v_2)+(v_3,v_3)+(v_4,v_4)-(v_1,v_1+v_2+v_4)-(v_2,-v_1)-(v_3,-v_2)$, and the remaining vectors are linearly independent and give the Gram matrix
\begin{equation*}G_K=
\begin {pmatrix}
4&-2&0&0&3&-3&1\\
-2&4&-2&-2&-1&3&-3\\
0&-2&4&0&-1&-1&3\\
0&-2&0&4&1&-1&1\\
3&-1&-1&1&4&-2&0\\
-3&3&-1&-1&-2&4&-2\\
1&-3&3&1&0&-2&4
\end{pmatrix},
\end{equation*}
which has determinant 64 and Smith invariants 1, 1, 2, 2, 2, 2, 4.  Using the same techniques as above, we may embed $K$ with index 8 in an integral lattice of determinant 1, which must be $\mathbb{Z}^7$.

\subsection{Negative of the Above Element}
Let $h$ be the negative of the element from the previous section. We have $hv_1=-v_1-v_2-v_4$, $hv_2=v_1$, $hv_3=v_2$, and $hv_4=-v_1-v_2-v_3$.  Thus, $K$ is spanned by $\{ (v_1,v_1),(v_2,v_2),(v_3,v_3),(v_4,v_4),(v_1,-v_1-v_2-v_4),(v_2,v_1),(v_3,v_2),(v_4,-v_1-v_2-v_3)\}$.  We may write $(v_4,-v_1-v_2-v_3)=-(v_1,v_1)-(v_2,v_2)-(v_3,v_3)+(v_4,v_4)+(v_1,-v_1-v_2-v_4)+(v_2,v_1)+(v_3,v_2)$, and the remaining vectors are linearly independent and give the Gram matrix
\begin{equation*}G_K=
\begin {pmatrix}
4&-2&0&0&1&1&-1\\
-2&4&-2&-2&-1&1&1\\
0&-2&4&0&1&-1&1\\
0&-2&0&4&-1&-1&-1\\ 
1&-1&1&-1&4&-2&0\\
1&1&-1&-1&-2&4&-2\\
-1&1&1&-1&0&-2&4
\end{pmatrix},
\end{equation*}
which has determinant 192 and Smith invariants 1, 1, 2, 2, 2, 2, 12.

\subsection{Element of Order 8}
Let $h$ be the element that sends $e_1$ to $-e_2$, $e_2$ to $e_3$, $e_3$ to $e_4$, and $e_4$ to $e_1$.  We have $hv_1=-v_2-v_3-v_4$, $hv_2=v_3$, $hv_3=-v_1-v_2-v_3$, and $hv_4=v_1+v_2+v_4$.  Thus, $K$ is spanned by $\{ (v_1,v_1),(v_2,v_2),(v_3,v_3),(v_4,v_4),(v_1,-v_2-v_3-v_4),(v_2,v_3),(v_3,-v_1-v_2-v_3),(v_4,v_1+v_2+v_4)\}$, and these vectors give the Gram matrix
\begin{equation*}G_K=
\begin {pmatrix}
4&-2&0&0&3&-1&-1&1\\
-2&4&-2&-2&-1&1&-1&-1\\
0&-2&4&0&-1&1&1&-1\\
0&-2&0&4&-1&-1&1&3\\
3&-1&-1&-1&4&-2&0&0\\
-1&1&1&-1&-2&4&-2&-2\\
-1&-1&1&1&0&-2&4&0\\
1&-1&-1&3&0&-2&0&4
\end{pmatrix},
\end{equation*}
which has determinant 64 and Smith invariants 1, 1, 2, 2, 2, 2, 2.  $K$ is thus a sublattice of index 2 in $D_4\perp D_4$.

\subsection{Element of Order 12}
Let $h$ be the element given by the matrix
\begin{equation*}
\begin {pmatrix}
1/2&1/2&-1/2&1/2\\
1/2&-1/2&1/2&1/2\\
1/2&1/2&1/2&-1/2\\
1/2&-1/2&-1/2&-1/2
\end{pmatrix}.
\end{equation*}
We have $hv_1=v_2+v_4$, $hv_2=v_1$, $hv_3=-v_1-v_2$, and $hv_4=v_2+v_3$, and $K$ is spanned by $\{ (v_1,v_1),(v_2,v_2),(v_3,v_3),(v_4,v_4),(v_1,v_2+v_4),(v_2,v_1),(v_3,-v_1-v_2),(v_4,v_2+v_3)\}$.  These vectors have the Gram matrix
\begin{equation*}G_K=
\begin {pmatrix}
4&-2&0&0&1&1&-1&-1\\
-2&4&-2&-2&0&1&-2&0\\
0&-2&4&0&-1&-1&3&1\\
0&-2&0&4&1&-1&1&1\\
1&0&-1&1&4&-2&0&0\\
1&1&-1&-1&-2&4&-2&-2\\
-1&-2&3&1&0&-2&4&0\\
-1&0&1&1&0&-2&0&4
\end{pmatrix},
\end{equation*}
which has determinant 16.  Thus, $K$ is the whole lattice $D_4\perp D_4$.

\section{$\mathbb{Z}^5$ and $D_5$}
We consider $\mathbb{Z}^5$ and its even sublattice $D_5$ together.  We use $K$ to refer to the lattice obtained from $\mathbb{Z}^5$ and $K'$ to refer to the sublattice obtained from $D_5$. The isometry group of both of these lattices is the group of signed permutation matrices $\mathbb{Z}_2 \wr Sym_5$, which has order $2^5\cdot 5!=3840$.  We use the standard basis $\{v_1=e_1-e_2,v_2=e_2-e_3,v_3=e_3-e_4,v_4=e_4-e_5,v_5=e_4+e_5\}$ for $D_5$.  We can make use of Theorem \ref{ncyclethm} to determine the lattices created by applying a permutation to $\mathbb{Z}^5$. 

\subsection{Transposition}
Let $h$ be the transposition $(e_1,e_2)$.  Then $h$ preserves the orthogonal decomposition $\mathbb{Z}^5=span(e_1,e_2) \perp span (e_3,e_4,e_5)$.  By Theorem \ref{ncyclethm}, the sublattice corresponding to the span of $e_1$ and $e_2$ is isometric to $A_3$, while the span of $e_3$, $e_4$, and $e_5$ is fixed by $h$, so the corresponding lattice is $\sqrt{2}\mathbb{Z}^3={A_1}^3$.  Thus for $\mathbb{Z}^5$ we have $K \cong A_3 \perp {A_1}^3$.  The sublattice $K'$ of index 2 corresponding to $D_5$ is spanned by $\{(v_1,v_1),(v_2,v_2),(v_3,v_3),(v_4,v_4),(v_5,v_5),(v_1,-v_1),(v_2,v_1+v_2)\}$.  We have $(0,v_1)=(v_2,v_1+v_2)-(v_2,v_2)$ and $(v_1,0)=(v_1,v_1)-(0,v_1)$, and thus we can see that $K'$ is spanned by $\{(v_1,0),(0,v_1),(v_2,v_2),(v_3,v_3),(v_4,v_4),(v_5,v_5)\}$.  This gives the Gram matrix
\begin{equation*}
G_{K'}=
\begin{pmatrix}
2&0&-1&0&0&0\\
0&2&-1&0&0&0\\
-1&-1&4&-2&0&0\\
0&0&-2&4&-2&-2\\
0&0&0&-2&4&0\\
0&0&0&-2&0&4
\end{pmatrix},
\end{equation*}
which has determinant 128 and Smith invariants 1, 1, 2, 2, 4, 8.  We can see from this construction that $K'$ contains a sublattice $J\cong\sqrt{2}D_5\perp\mathbb{Z}(v_1,-v_1)$ with index 2 (since $J$ has determinant 512), and we can form $K'$ from $J$ using the glue vector $(v_1,0)$.

\subsection{3-cycle}
Let $h$ be the 3-cycle $(e_1,e_2,e_3)$.  Then by Theorem \ref{ncyclethm}, the lattice $K$ obtained from $\mathbb{Z}^5$ is isometric to $A_5\perp {A_1}^2$.  The sublattice $K'$ is spanned by $\{(v_1,v_1),(v_2,v_2),(v_3,v_3),(v_4,v_4),(v_5,v_5),\\(v_1,v_2),(v_2,-v_1-v_2),(v_3,v_1+v_2+v_3)\}$.  We may write $(v_1,v_2)=(v_1,v_1)+2(v_2,v_2)+3(v_3,v_3)-2(v_2,-v_1-v_2)-3(v_3,v_1+v_2+v_3)$, and the remaining vectors are linearly independent and give the Gram matrix
\begin{equation*}
G_{K'}=
\begin{pmatrix}
4&-2&0&0&0&-2&1\\
-2&4&-2&0&0&1&-1\\
0&-2&4&-2&-2&0&3\\
0&0&-2&4&0&0&-2\\
0&0&-2&0&4&0&-2\\
-2&1&0&0&0&4&-2\\
1&-1&3&-2&-2&-2&4
\end{pmatrix},
\end{equation*}
which has determinant 96 and the Smith invariants 1, 1, 1, 1, 2, 2, 24.

\subsection{4-cycle}
Let $h$ be the 4-cycle $(e_1,e_2,e_3,e_4)$.  Then by Theorem \ref{ncyclethm}, the lattice $K$ obtained from $\mathbb{Z}^5$ is isometric to $A_7\perp A_1$.  The sublattice $K'$ is spanned by $\{(v_1,v_1),(v_2,v_2),(v_3,v_3),(v_4,v_4),(v_5,v_5),\\(v_1,v_2),(v_2,v_3),(v_3,-v_1-v_2-v_3),(v_4,v_1+v_2+v_3+v_4),(v_5,v_1+v_2+v_3+v_5)\}$.  We can write
\begin{align}
(v_1,v_2)=&(v_1,v_1)+2(v_2,v_2)+3(v_3,v_3)+4(v_4,v_4)\nonumber\\
&{}-2(v_2,v_3)-3(v_3,-v_1-v_2-v_3)-4(v_4,v_1+v_2+v_3+v_4)\nonumber\\
(v_5,v_1+v_2+v_3+v_5)=&-(v_4,v_4)+(v_5,v_5)+(v_4,v_1+v_2+v_3+v_4),\nonumber
\end{align}
\begin{equation*}
G_{K'}=
\begin{pmatrix}
4&-2&0&0&0&-1&-1&1\\
-2&4&-2&0&0&1&-1&0\\
0&-2&4&-2&-2&1&1&-1\\
0&0&-2&4&0&-1&0&3\\
0&0&-2&0&4&-1&0&-1\\
-1&1&1&-1&-1&4&-2&0\\
-1&-1&1&0&0&-2&4&-2\\
1&0&-1&3&-1&0&-2&4
\end{pmatrix},
\end{equation*}
which has determinant 64 and Smith invariants 1, 1, 1, 1, 1, 1, 8, 8.

\subsection{5-cycle}
Let $h$ be the 5-cycle $(e_1,e_2,e_3,e_4,e_5)$.  Then by Theorem \ref{ncyclethm}, the lattice $K$ obtained from $\mathbb{Z}^5$ is isometric to $A_9$.  The sublattice $K'$ is spanned by $\{(v_1,v_1),(v_2,v_2),(v_3,v_3),(v_4,v_4),(v_5,v_5),\\(v_1,v_2),(v_2,v_3),(v_3,v_4),(v_4,-v_1-v_2-v_3-v_4),(v_5,v_1+v_2+v_3+v_5)\}$.  We may write
\begin{align}
(v_1,0)=&(v_1,v_1)+(v_2,v_2)+2(v_3,v_3)+(v_4,v_4)+2(v_5,v_5)\nonumber\\
&{}-(v_2,v_3)-2(v_3,v_4)-(v_4,-v_1-v_2-v_3-v_4)-2(v_5,v_1+v_2+v_3+v_5)\nonumber\\
(0,v_1)=&(v_1,v_1)-(v_1,0)\nonumber\\
(0,v_2)=&(v_1,v_2)-(v_1,0)\nonumber\\
(v_2,0)=&(v_2,v_2)-(0,v_2)\nonumber\\
(0,v_3)=&(v_2,v_3)-(v_2,0)\nonumber\\
(v_3,0)=&(v_3,v_3)-(0,v_3)\nonumber\\
(0,v_4)=&(v_3,v_4)-(v_3,0)\nonumber\\
(v_4,0)=&(v_4,v_4)-(0,v_4),\nonumber
\end{align}
and hence $K'$ is spanned by $\{(v_1,0),(0,v_1),(v_2,0),(0,v_2),(v_3,0),(0,v_3),(v_4,0),(0,v_4),(v_5,v_5)\}$, and these vectors give the Gram matrix
\begin{equation*}
G_{K'}=
\begin{pmatrix}
2&0&-1&0&0&0&0&0&0\\
0&2&0&-1&0&0&0&0&0\\
-1&0&2&0&-1&0&0&0&0\\
0&-1&0&2&0&-1&0&0&0\\
0&0&-1&0&2&0&-1&0&-1\\
0&0&0&-1&0&2&0&-1&-1\\
0&0&0&0&-1&0&2&0&0\\
0&0&0&0&0&-1&0&2&0\\
0&0&0&0&-1&-1&0&0&4
\end{pmatrix},
\end{equation*}
which has Smith invariants 1, 1, 1, 1, 1, 1, 1, 1, 40.  $K'$ contains a sublattice $J\cong\sqrt{2}D_5\perp\sqrt{2}A_4$ with index 16, and the glue vectors are $(v_1,0),(v_2,0),(v_3,0),(v_4,0)$.

\subsection{Product of Disjoint Transpositions}
Let $h$ be the product of disjoint transpositions $(e_1,e_2)(e_3,e_4)$.  Then by Theorem \ref{ncyclethm}, the lattice $K$ obtained from $\mathbb{Z}^5$ is isometric to $A_3\perp A_3\perp A_1$.  The sublattice $K'$ is spanned by $\{(v_1,v_1),(v_2,v_2),(v_3,v_3),(v_4,v_4),(v_5,v_5),(v_1,-v_1),(v_2,v_1+v_2+v_3),(v_3,-v_3),(v_4,v_3+v_4),(v_5,v_3+v_5)\}$.  We can see that $K'$ is spanned by $\{(v_1,0),(0,v_1),(v_2,v_2),(v_3,0),(0,v_3),\\(v_4,v_4),(v_5,v_5)\}$, and these vectors give the Gram matrix
\begin{equation*}
G_{K'}=
\begin{pmatrix}
2&0&-1&0&0&0&0\\
0&2&-1&0&0&0&0\\
-1&-1&4&-1&-1&0&0\\
0&0&-1&2&0&-1&-1\\
0&0&-1&0&2&-1&-1\\
0&0&0&-1&-1&4&0\\
0&0&0&-1&-1&0&4
\end{pmatrix},
\end{equation*}
which has Smith invariants 1, 1, 1, 1, 4, 4, 8.  $K'$ contains the sublattice $J\cong\sqrt{2}D_5\perp\mathbb{Z}(v_1,-v_1)\perp\mathbb{Z}(v_3,-v_3)$ with index 4, and the glue vectors are $(v_1,0),(v_3,0)$.

\subsection{Product of Transposition and 3-Cycle}
Let $h$ be the permutation $(e_1,e_2,e_3)(e_4,e_5)$.  Then by Theorem \ref{ncyclethm}, the lattice $K$ obtained from $\mathbb{Z}^5$ is isometric to $A_5\perp A_3$.  The sublattice $K'$ is spanned by $\{(v_1,v_1),(v_2,v_2),(v_3,v_3),\\(v_4,v_4),(v_5,v_5),(v_1,v_2),(v_2,-v_1-v_2),(v_3,v_1+v_2+v_3+v_4),(v_4,-v_4)\}$.  We may write
\begin{align}
(v_1,0)&=(v_1,v_1)+(v_2,v_2)+2(v_3,v_3)+(v_4,v_4)-(v_2,-v_1-v_2)-2(v_3,v_1+v_2+v_3+v_4)-(v_4,-v_4)\nonumber\\
(0,v_1)&=(v_1,v_1)-(v_1,0)\nonumber\\
(0,v_2)&=(v_1,v_2)-(v_1,0)\nonumber\\
(v_2,0)&=(v_2,v_2)-(0,v_2)\nonumber\\
(0,v_4)&=(v_3,v_1+v_2+v_3+v_4)-(v_3,v_3)-(0,v_1)-(0,v_2)\nonumber\\
(v_4,0)&=(v_4,v_4)=(0,v_4),\nonumber
\end{align}
so $K'$ is spanned by $\{(v_1,0),(0,v_1),(v_2,0),(0,v_2),(v_3,v_3),(v_4,0),(0,v_4),(v_5,v_5)\}$, and these vectors give the Gram matrix
\begin{equation*}
G_{K'}=
\begin{pmatrix}
2&0&-1&0&0&0&0&0\\
0&2&0&-1&0&0&0&0\\
-1&0&2&0&-1&0&0&0\\
0&-1&0&2&-1&0&0&0\\
0&0&-1&-1&4&-1&-1&-2\\
0&0&0&0&-1&2&0&0\\
0&0&0&0&-1&0&2&0\\
0&0&0&0&-2&0&0&4
\end{pmatrix},
\end{equation*}
which has Smith invariants 1, 1, 1, 1, 1, 1, 4, 24.  $K'$ contains a sublattice $J\cong\sqrt{2}D_5\perp\sqrt{2}A_2\perp\mathbb{Z}(v_4,-v_4)$ with index 8, and the glue vectors are $(v_1,0),(v_2,0),(v_4,0)$.

\subsection{Negative of Transposition}
Let $h$ be the negative of the transposition $(e_1,e_2)$.  Note that negation at a coordinate gives two orthogonal vectors of norm 2, so lattices corresponding to the spans of $e_3$, $e_4$, and $e_5$ are each isometric to ${A_1}^2$.  Further, the restriction of $h$ to $span(e_1,e_2)$ is conjugate to the positive transposition, since if we let $g$ negate $e_1$ and fix all other coordinates, then $g^{-1}hg$ is the transposition $(e_1,e_2)$.  A similar argument holds for all cycles of even length if we let $g$ negate all odd-numbered coordinates.  Thus the lattice $K$ obtained from $\mathbb{Z}^5$ is $A_3 \perp {A_1}^6$.   The sublattice $K'$ corresponding to $D_5$ is spanned by $\{(v_1,v_1),(v_2,v_2),(v_3,v_3),(v_4,v_4),(v_5,v_5),(v_2,-v_1-v_2),(v_3,-v_3),(v_4,-v_4),(v_5,-v_5)\}$.  These vectors give the Gram matrix
\begin{equation*}
G_{K'}=
\begin{pmatrix}
4&-2&0&0&0&-2&0&0&0\\
-2&4&-2&0&0&1&0&0&0\\
0&-2&4&-2&-2&0&0&0&0\\
0&0&-2&4&0&0&0&0&0\\
0&0&-2&0&4&0&0&0&0\\
-2&1&0&0&0&4&-2&0&0\\
0&0&0&0&0&-2&4&-2&-2\\
0&0&0&0&0&0&-2&4&0\\
0&0&0&0&0&0&-2&0&4
\end{pmatrix},
\end{equation*}
which has determinant 4096 and Smith invariants 1, 1, 2, 2, 2, 2, 4, 8, 8.

\subsection{Negative of Product of Disjoint Transpositions}
Let $h$ be the negative of the product of disjoint transpositions $(e_1,e_2)(e_3,e_4)$.  Then the lattice $K$ obtained from $\mathbb{Z}^5$ is isometric to ${A_3}^2 \perp {A_1}^2$.  The sublattice $K'$ corresponding to $D_5$ is spanned by $\{(v_1,v_1),(v_2,v_2),(v_3,v_3),(v_4,v_4),(v_5,v_5),(v_2,-v_1-v_2-v_3),(v_4,-v_3-v_4),(v_5,-v_3-v_5)\}$. These vectors give the Gram matrix
\begin{equation*}
G_{K'}=
\begin{pmatrix}
4&-2&0&0&0&-2&0&0\\
-2&4&-2&0&0&2&1&1\\
0&-2&4&-2&-2&-2&-2&-2\\
0&0&-2&4&0&1&1&1\\
0&0&-2&0&4&1&1&1\\
-2&2&-2&1&1&4&0&0\\
0&1&-2&1&1&0&4&0\\
0&1&-2&1&1&0&0&4
\end{pmatrix},
\end{equation*}
which has determinant 1024 and Smith invariants 1, 1, 1, 1, 4, 4, 8, 8.

\subsection{Negative of 4-Cycle}
Let $h$ be the negative of the 4-cycle $(e_1,e_2,e_3,e_4)$.  Then the lattice $K$ obtained from $\mathbb{Z}^5$ is isometric to $A_7 \perp {A_1}^2$.  The sublattice $K'$ is spanned by $\{(v_1,v_1),(v_2,v_2),(v_3,v_3),(v_4,v_4),\\(v_5,v_5),(v_1,-v_2),(v_2,-v_3),(v_3,v_1+v_2+v_3),(v_4,-v_1-v_2-v_3-v_4),(v_5,-v_1-v_2-v_3-v_5)\}$.  We can write $(v_3,v_1+v_2+v_3)=(v_1,v_1)+(v_3,v_3)-(v_1,-v_2)$, and the remaining vectors give the Gram matrix
\begin{equation*}
G_{K'}=
\begin{pmatrix}
4&-2&0&0&0&3&-1&-1&-1\\
-2&4&-2&0&0&-3&3&0&0\\
0&-2&4&-2&-2&1&-3&-1&-1\\
0&0&-2&4&0&0&1&1&1\\
0&0&-2&0&4&0&1&1&1\\
3&-3&1&0&0&4&-2&0&0\\
-1&3&-3&1&1&-2&4&0&0\\
-1&0&-1&1&1&0&0&4&0\\
-1&0&-1&1&1&0&0&0&4
\end{pmatrix},
\end{equation*}
which has determinant 512 and Smith invariants 1, 1, 1, 1, 1, 1, 8, 8, 8.

\subsection{Negative of 3-Cycle}
Let $h$ be the negative of the 3-cycle $(e_1,e_2,e_3)$.  If we consider only the span of the first three unit vectors, the corresponding lattice is spanned by $\{e_1+e_6, e_2+e_7, e_3+e_8, e_1-e_7, e_2-e_8, e_3-e_6\}$.  These vectors give the Gram matrix
\begin{equation*}
\begin{pmatrix}
2&0&0&1&0&-1\\
0&2&0&-1&1&0\\
0&0&2&0&-1&1\\
1&-1&0&2&0&0\\
0&1&-1&0&2&0\\
-1&0&1&0&0&2
\end{pmatrix},
\end{equation*}
which has determinant 4.  Since all the vectors in this lattice have even coordinate sum, this is a sublattice of $D_6$, so since it has determinant 4, it must be all of $D_6$.  Hence the lattice $K$ obtained from $\mathbb{Z}^5$ is isometric to $D_6 \perp {A_1}^4$.  The sublattice $K'$ is spanned by $\{(v_1,v_1),(v_2,v_2),(v_3,v_3),(v_4,v_4),(v_5,v_5),(v_1,-v_2),(v_2,v_1+v_2),(v_3,-v_1-v_2-v_3),(v_4,-v_4),(v_5,-v_5)\}$, giving the Gram matrix
\begin{equation*}
G_{K'}=
\begin{pmatrix}
4&-2&0&0&0&3&0&-1&0&0\\
-2&4&-2&0&0&-3&3&-1&0&0\\
0&-2&4&-2&-2&1&-2&1&0&0\\
0&0&-2&4&0&0&0&0&0&0\\
0&0&-2&0&4&0&0&0&0&0\\
3&-3&1&0&0&4&-2&0&0&0\\
0&3&-2&0&0&-2&4&-2&0&0\\
-1&-1&1&0&0&0&-2&4&-2&-2\\
0&0&0&0&0&0&0&-2&4&0\\
0&0&0&0&0&0&0&-2&0&4
\end{pmatrix},
\end{equation*}
which has determinant 1024 and Smith invariants 1, 1, 1, 1, 2, 2, 2, 2, 8, 8.

\subsection{Negative of Product of 3-Cycle and Transposition}
Let $h$ be the negative of the permutation $(e_1,e_2,e_3)(e_4,e_5)$.  Then the lattice $K$ obtained from $\mathbb{Z}^5$ is isometric to $D_6\perp A_3$.  The sublattice $K'$ is spanned by $\{(v_1,v_1),(v_2,v_2),(v_3,v_3),\\(v_4,v_4),(v_5,v_5),(v_1,-v_2),(v_2,v_1+v_2),(v_3,-v_1-v_2-v_3-v_4),(v_5,-v_5)\}$.  These vectors give the Gram matrix
\begin{equation*}
G_{K'}=
\begin{pmatrix}
4&-2&0&0&0&3&0&-1&0\\
-2&4&-2&0&0&-3&3&-1&0\\
0&-2&4&-2&-2&1&-2&2&0\\
0&0&-2&4&0&0&0&-2&0\\
0&0&-2&0&4&0&0&0&0\\
3&-3&1&0&0&4&-2&0&0\\
0&3&-2&0&0&-2&4&-2&0\\
-1&-1&2&-2&0&0&-2&4&-2\\
0&0&0&0&0&0&0&-2&4
\end{pmatrix},
\end{equation*}
which has determinant 256 and Smith invariants 1, 1, 1, 1, 1, 1, 4, 8, 8.

\subsection{Negative of 5-Cycle}
Let $h$ be the negative of the 5-cycle $(e_1,e_2,e_3,e_4,e_5)$.  Then the lattice $K$ obtained from $\mathbb{Z}^5$ has the Gram matrix
\begin{equation*}
G_{K}=
\begin{pmatrix}
2&0&0&0&0&1&0&0&0&-1\\
0&2&0&0&0&-1&1&0&0&0\\
0&0&2&0&0&0&-1&1&0&0\\
0&0&0&2&0&0&0&-1&1&0\\
0&0&0&0&2&0&0&0&-1&1\\
1&-1&0&0&0&2&0&0&0&0\\
0&1&-1&0&0&0&2&0&0&0\\
0&0&1&-1&0&0&0&2&0&0\\
0&0&0&1&-1&0&0&0&2&0\\
-1&0&0&0&1&0&0&0&0&2
\end{pmatrix},
\end{equation*}
which has determinant 4.  Since $K$ is an even sublattice of $\mathbb{Z}^{10}$ with determinant 4, $K\cong D_{10}$.  The sublattice $K'$ is spanned by $\{(v_1,v_1),(v_2,v_2),(v_3,v_3),(v_4,v_4),(v_5,v_5),(v_1,-v_2),(v_2,-v_3),\\(v_3,-v_4),(v_4,v_1+v_2+v_3+v_4),(v_5,-v_1-v_2-v_3-v_5)\}$.  These vectors give the Gram matrix
\begin{equation*}
G_{K'}=
\begin{pmatrix}
4&-2&0&0&0&3&-1&0&1&-1\\
-2&4&-2&0&0&-3&3&-1&0&0\\
0&-2&4&-2&-2&1&-3&3&-1&-1\\
0&0&-2&4&0&0&1&-3&3&1\\
0&0&-2&0&4&0&1&-1&-1&1\\
3&-3&1&0&0&4&-2&0&0&0\\
-1&3&-3&1&1&-2&4&-2&0&0\\
0&-1&3&-3&-1&0&-2&4&-2&-2\\
1&0&-1&3&-1&0&0&-2&4&0\\
-1&0&-1&1&1&0&0&-2&0&4
\end{pmatrix},
\end{equation*}
which has determinant 64 and Smith invariants 1, 1, 1, 1, 1, 1, 1, 1, 8, 8.

\subsection{Negation of One Coordinate}
Let $h$ negate $e_1$ and fix all other coordinates.  Then the lattice $K$ obtained from $\mathbb{Z}^5$ is isometric to ${A_1}^6$.  The sublattice $K'$ is isometric to $\sqrt{2}D_6$ by Theorem \ref{DnNegations}.

\subsection{Negation of Two Coordinates}
Let $h$ negate $e_1$ and $e_2$ and fix all other coordinates.  Then the lattice $K$ obtained from $\mathbb{Z}^5$ is isometric to ${A_1}^7$.  Also, by Theorem \ref{DnNegations}, $K' \cong \sqrt{2}D_7$.

\subsection{Negation of Three Coordinates}
Let $h$ negate three coordinates.  Then we have $K\cong {A_1}^8$ and $K'\cong \sqrt{2}D_8$.

\subsection{Negation of Four Coordinates}
Let $h$ negate four coordinates.  Then we have $K\cong {A_1}^9$ and $K'\cong \sqrt{2}D_9$.

\subsection{Transposition with a Negation} \label{transneg}
Let $h$ send $e_1$ to $-e_2$, $e_2$ to $e_1$, and fix the remaining coordinates.  From section \ref{Z2-90}, we know that the sublattice corresponding to the span of the first two coordinates is isometric to $D_4$.  Thus $K\cong D_4\perp {A_1}^3$.  The sublattice $K'$ is spanned by $\{(v_1,v_1),(v_2,v_2),(v_3,v_3),(v_4,v_4),\\(v_5,v_5),(v_1,-v_1-2v_2-2v_3-v_4-v_5),(v_2,v_1+v_2)\}$.  These vectors give the Gram matrix
\begin{equation*}
G_{K'}=
\begin{pmatrix}
4&-2&0&0&0&2&0\\
-2&4&-2&0&0&-2&3\\
0&-2&4&-2&-2&0&-2\\
0&0&-2&4&0&0&0\\
0&0&-2&0&4&0&0\\
2&-2&0&0&0&4&-2\\ 
0&3&-2&0&0&-2&4
\end{pmatrix},
\end{equation*}
which has determinant 128 and Smith invariants 1, 1, 2, 2, 2, 2, 8.

\subsection{Transposition and Negation of One Disjoint Coordinate}
Let $h$ transpose $e_1$ and $e_2$, negate $e_3$, and fix the other coordinates.  Then the lattice $K$ obtained from $\mathbb{Z}^5$ is isometric to $A_3 \perp {A_1}^4$.  The sublattice $K'$ is spanned by $\{(v_1,v_1),(v_2,v_2),(v_3,v_3),(v_4,v_4),(v_5,v_5),(v_1,-v_1),(v_2,v_1+v_2+2v_3+v_4+v_5),(v_3,-v_3-v_4-v_5)\}$.  We may write $(v_1,-v_1)=(v_1,v_1)+2(v_2,v_2)+2(v_3,v_3)-2(v_2,v_1+v_2+2v_3+v_4+v_5)-2(v_3,-v_3-v_4-v_5)$, and the remaining vectors give the Gram matrix
\begin{equation*}
G_{K'}=
\begin{pmatrix}
4&-2&0&0&0&0&0\\
-2&4&-2&0&0&1&0\\
0&-2&4&-2&-2&0&2\\
0&0&-2&4&0&0&-2\\
0&0&-2&0&4&0&-2\\
0&1&0&0&0&4&-2\\
0&0&2&-2&-2&-2&4
\end{pmatrix},
\end{equation*}
which has determinant 256 and Smith invariants 1, 1, 2, 2, 4, 4, 4.

\subsection{Transposition and Negation of Two Disjoint Coordinates}
Let $h$ transpose $e_1$ and $e_2$, negate $e_3$ and $e_4$, and fix $e_5$.  Then the lattice $K$ obtained from $\mathbb{Z}^5$ is isometric to $A_3 \perp {A_1}^5$.  The sublattice $K'$ is spanned by $\{(v_1,v_1),(v_2,v_2),(v_3,v_3),(v_4,v_4),\\(v_5,v_5),(v_1,-v_1),(v_2,v_1+v_2+2v_3+v_4+v_5),(v_3,-v_3),(v_4,-v_5),(v_5,-v_4)\}$.  We may write $(v_1,-v_1)=(v_1,v_1)+2(v_2,v_2)+2(v_3,v_3)+2(v_4,v_4)-2(v_2,v_1+v_2+2v_3+v_4+v_5)-2(v_3,-v_3)-2(v_4,-v_5)$ and $(v_5,-v_4)=-(v_4,v_4)+(v_5,v_5)+(v_4,-v_5)$, and the remaining vectors give the Gram matrix
\begin{equation*}
G_{K'}=
\begin{pmatrix}
4&-2&0&0&0&0&0&0\\
-2&4&-2&0&0&1&0&0\\
0&-2&4&-2&-2&0&0&0\\
0&0&-2&4&0&0&0&2\\
0&0&-2&0&4&0&0&-2\\
0&1&0&0&0&4&-2&0\\
0&0&0&0&0&-2&4&-2\\
0&0&0&2&-2&0&-2&4
\end{pmatrix},
\end{equation*}
which has determinant 512 and Smith invariants 1, 1, 2, 2, 2, 2, 4, 8.

\subsection{Transposition with Negation and Negation of One Disjoint Coordinate}
Let $h$ send $e_1$ to $-e_2$, $e_2$ to $e_1$, negate $e_3$, and fix the remaining coordinates.   Then $K\cong D_4\perp {A_1}^4$.  The sublattice $K'$ is spanned by $\{(v_1,v_1),(v_2,v_2),(v_3,v_3),(v_4,v_4),(v_5,v_5),(v_1,-v_1-2v_2-2v_3-v_4-v_5),(v_2,v_1+v_2+2v_3+v_4+v_5),(v_3,-v_3-v_4-v_5)\}$.  These vectors give the Gram matrix
\begin{equation*}
G_{K'}=
\begin{pmatrix}
4&-2&0&0&0&2&0&0\\
-2&4&-2&0&0&-2&1&0\\
0&-2&4&-2&-2&0&0&2\\
0&0&-2&4&0&0&0&-2\\
0&0&-2&0&4&0&0&-2\\
2&-2&0&0&0&4&-2&0\\
0&1&0&0&0&-2&4&-2\\
0&0&2&-2&-2&0&-2&4
\end{pmatrix},
\end{equation*}
which has determinant 256 and Smith invariants 1, 1, 2, 2, 2, 2, 4, 4.

\subsection{Transposition with Negation and Negation of Two Disjoint Coordinates}
Let $h$ send $e_1$ to $-e_2$, $e_2$ to $e_1$, negate $e_3$ and $e_4$, and fix $e_5$.   Then $K\cong D_4\perp {A_1}^5$.  The sublattice $K'$ is spanned by $\{(v_1,v_1),(v_2,v_2),(v_3,v_3),(v_4,v_4),(v_5,v_5),(v_1,-v_1-2v_2-2v_3-v_4-v_5),(v_2,v_1+v_2+2v_3+v_4+v_5),(v_3,-v_3),(v_4,-v_5),(v_5,-v_4)\}$.  We have $(v_5,-v_4)=-(v_4,v_4)+(v_5,v_5)+(v_4,-v_5)$, and the remaining vectors give the Gram matrix
\begin{equation*}
G_{K'}=
\begin{pmatrix}
4&-2&0&0&0&2&0&0&0\\
-2&4&-2&0&0&-2&1&0&0\\
0&-2&4&-2&-2&0&0&0&0\\
0&0&-2&4&0&0&0&0&2\\
0&0&-2&0&4&0&0&0&-2\\
2&-2&0&0&0&4&-2&0&0\\
0&1&0&0&0&-2&4&-2&0\\
0&0&0&0&0&0&-2&4&-2\\
0&0&0&2&-2&0&0&-2&4
\end{pmatrix},
\end{equation*}
which has determinant 512 and Smith invariants 1, 1, 2, 2, 2, 2, 2, 2, 8.

\subsection{Transposition with Negation and Negation of Three Disjoint Coordinates}
Let $h$ be the negative of the element in section \ref{transneg}.  We have $K\cong D_4\perp {A_1}^6$.  The sublattice $K'$ is spanned by $\{(v_1,v_1),(v_2,v_2),(v_3,v_3),(v_4,v_4),(v_5,v_5),(v_1,v_1+2v_2+2v_3+v_4+v_5),(v_2,-v_1-v_2),(v_3,-v_3),(v_4,-v_4),(v_5,-v_5)\}$.  These vectors give the Gram matrix
\begin{equation*}
G_{K'}=
\begin{pmatrix}
4&-2&0&0&0&2&-2&0&0&0\\
-2&4&-2&0&0&0&1&0&0&0\\
0&-2&4&-2&-2&0&0&0&0&0\\
0&0&-2&4&0&0&0&0&0&0\\
0&0&-2&0&4&0&0&0&0&0\\ 
2&0&0&0&0&4&-2&0&0&0\\
-2&1&0&0&0&-2&4&-2&0&0\\
0&0&0&0&0&0&-2&4&-2&-2\\
0&0&0&0&0&0&0&-2&4&0\\
0&0&0&0&0&0&0&-2&0&4
\end{pmatrix},
\end{equation*}
which has determinant 4096 and Smith invariants 1, 1, 2, 2, 2, 2, 2, 2, 8, 8.

\subsection{Product of Transposition with Negation and Transposition}
Let $h$ send $e_1$ to $-e_2$, $e_2$ to $e_1$, and switch $e_3$ and $e_4$.  Then we have $K\cong D_4 \perp A_3 \perp A_1$.  The sublattice $K'$ is spanned by $\{(v_1,v_1),(v_2,v_2),(v_3,v_3),(v_4,v_4),(v_5,v_5),(v_1,-v_1-2v_2-2v_3-v_4-v_5),(v_2,v_1+v_2+v_3),(v_3,-v_3),(v_4,v_3+v_4),(v_5,v_3+v_5)\}$.  We may write $(v_3,-v_3)=(v_3,v_3)+2(v_4,v_4)-2(v_4,v_3+v_4)$ and $(v_5,v_3+v_5)=-(v_4,v_4)+(v_5,v_5)+(v_4,v_3+v_4)$, and the other vectors give the Gram matrix
\begin{equation*}
G_{K'}=
\begin{pmatrix}
4&-2&0&0&0&2&0&0\\
-2&4&-2&0&0&-2&2&-1\\
0&-2&4&-2&-2&0&0&0\\
0&0&-2&4&0&0&-1&3\\
0&0&-2&0&4&0&-1&-1\\
2&-2&0&0&0&4&-2&0\\
0&2&0&-1&-1&-2&4&0\\ 
0&-1&0&3&-1&0&0&4
\end{pmatrix},
\end{equation*}
which has determinant 128 and Smith invariants 1, 1, 1, 1, 2, 2, 4, 8.

\subsection{Negative of the Above}
Let $h$ be the negative of the above element.  Then $K \cong D_4 \perp A_3 \perp {A_1}^2$.  
The sublattice $K'$ is spanned by $\{(v_1,v_1),(v_2,v_2),(v_3,v_3),(v_4,v_4),(v_5,v_5),(v_1,v_1+2v_2+2v_3+v_4+v_5),(v_2,-v_1-v_2-v_3),(v_4,-v_3-v_4),(v_5,-v_3-v_5)\}$.  These vectors give the Gram matrix
\begin{equation*}
G_{K'}=
\begin{pmatrix}
4&-2&0&0&0&2&-2&0&0\\
-2&4&-2&0&0&0&2&1&1\\
0&-2&4&-2&-2&0&-2&-2&-2\\
0&0&-2&4&0&0&1&1&1\\
0&0&-2&0&4&0&1&1&1\\
2&0&0&0&0&4&-2&0&0\\
-2&2&-2&1&1&-2&4&0&0\\
0&1&-2&1&1&0&0&4&0\\
0&1&-2&1&1&0&0&0&4
\end{pmatrix},
\end{equation*}
which has determinant 1024 and Smith invariants 1, 1, 1, 1, 2, 2, 4, 8, 8.

\subsection{Product of Disjoint Transpositions with Negations}
Let $h$ send $e_1$ to $-e_2$, $e_2$ to $e_1$, and $e_3$ to $-e_4$, $e_4$ to $e_3$, and fix $e_5$.  Then $K \cong {D_4}^2 \perp A_1$.  The sublattice $K'$ is spanned by $\{(v_1,v_1),(v_2,v_2),(v_3,v_3),(v_4,v_4),(v_5,v_5),(v_1,-v_1-2v_2-2v_3-v_4-v_5),(v_2,v_1+v_2+v_3+v_4+v_5),(v_3,-v_3-v_4-v_5),(v_4,v_3+v_4),(v_5,v_3+v_5)\}$.  We may write $(v_5,v_3+v_5)=-(v_4,v_4)+(v_5,v_5)+(v_4,v_3+v_4)$, and the remaining vectors give the Gram matrix
\begin{equation*}
G_{K'}=
\begin{pmatrix}
4&-2&0&0&0&2&0&0&0\\
-2&4&-2&0&0&-2&2&0&-1\\
0&-2&4&-2&-2&0&-2&2&0\\
0&0&-2&4&0&0&1&-2&3\\
0&0&-2&0&4&0&1&-2&-1\\
2&-2&0&0&0&4&-2&0&0\\
0&2&-2&1&1&-2&4&-2&0\\ 
0&0&2&-2&-2&0&-2&4&-2\\
0&-1&0&3&-1&0&0&-2&4
\end{pmatrix},
\end{equation*}
which has determinant 128 and Smith invariants 1, 1, 1, 1, 2, 2, 2, 2, 8.

\subsection{Negative of the Above}
Let $h$ be the negative of the above element.  Then $K \cong {D_4}^2 \perp {A_1}^2$.  The sublattice $K'$ is spanned by $\{(v_1,v_1),(v_2,v_2),(v_3,v_3),(v_4,v_4),(v_5,v_5),(v_1,v_1+2v_2+2v_3+v_4+v_5),(v_2,-v_1-v_2-v_3-v_4-v_5),(v_3,v_3+v_4+v_5),(v_4,-v_3-v_4),(v_5,-v_3-v_5)\}$.  These vectors give the Gram matrix
\begin{equation*}
G_{K'}=
\begin{pmatrix}
4&-2&0&0&0&2&-2&0&0&0\\
-2&4&-2&0&0&0&2&-2&1&1\\
0&-2&4&-2&-2&0&0&2&-2&-2\\
0&0&-2&4&0&0&-1&0&1&1\\
0&0&-2&0&4&0&-1&0&1&1\\
2&0&0&0&0&4&-2&0&0&0\\
-2&2&0&-1&-1&-2&4&-2&0&0\\
0&-2&2&0&0&0&-2&4&-2&-2\\
0&1&-2&1&1&0&0&-2&4&0\\
0&1&-2&1&1&0&0&-2&0&4
\end{pmatrix},
\end{equation*}
which has determinant 1024 and Smith invariants 1, 1, 1, 1, 2, 2, 2, 2, 8, 8.

\subsection{Product of 3-Cycle and Transposition with Negation}
Let $h$ be act as the 3-cycle $(e_1,e_2,e_3)$ and send $e_4$ to $-e_5$ and $e_5$ to $e_4$.  Then  the lattice $K$ obtained from $\mathbb{Z}^5$ is isometric to $A_5\perp D_4$.  The sublattice $K'$ is spanned by $\{(v_1,v_1),(v_2,v_2),(v_3,v_3),(v_4,v_4),(v_5,v_5),(v_1,v_2),(v_2,-v_1-v_2),(v_3,v_1+v_2+v_3+v_5),(v_4,-v_5),(v_5,v_4)\}$.  We may write
\begin{align}
(v_1,0)=&(v_1,v_1)+(v_2,v_2)+2(v_3,v_3)+(v_4,v_4)+(v_5,v_5)-(v_2,-v_1-v_2)\nonumber\\&{}-2(v_3,v_1+v_2+v_3+v_5)-(v_4,-v_5)-(v_5,v_4)\nonumber\\
(0,v_1)=&(v_1,v_1)-(v_1,0)\nonumber\\
(0,v_2)=&(v_1,v_2)-(v_1,0)\nonumber\\
(v_2,0)=&(v_2,v_2)-(0,v_2)\nonumber\\
(0,v_5)=&(v_3,v_1+v_2+v_3+v_5)-(v_3,v_3)-(0,v_1)-(0,v_2)\nonumber\\
(v_5,0)=&(v_5,v_5)-(0,v_5)\nonumber\\
(0,v_4)=&(v_5,v_4)-(v_5,0)\nonumber\\
(v_4,0)=&(v_4,v_4)-(0,v_4),\nonumber\\
\end{align}
so $K'$ is spanned by $\{(v_1,0),(0,v_1),(v_2,0),(0,v_2),(v_3,v_3),(v_4,0),(0,v_4),(v_5,0),(0,v_5)\}$, and these vectors give the Gram matrix
\begin{equation*}
G_{K'}=
\begin{pmatrix}
2&0&-1&0&0&0&0&0&0\\ 
0&2&0&-1&0&0&0&0&0\\
-1&0&2&0&-1&0&0&0&0\\
0&-1&0&2&-1&0&0&0&0\\
0&0&-1&-1&4&-1&-1&-1&-1\\
0&0&0&0&-1&2&0&0&0\\
0&0&0&0&-1&0&2&0&0\\
0&0&0&0&-1&0&0&2&0\\
0&0&0&0&-1&0&0&0&2
\end{pmatrix},
\end{equation*}
which has determinant 96 and Smith invariants 1, 1, 1, 1, 1, 1, 2, 2, 24.

\subsection{Negative of the Above}
Let $h$ be the negative of the above element.  Then $K \cong D_6 \perp D_4$.  The sublattice $K'$ is spanned by $\{(v_1,v_1),(v_2,v_2),(v_3,v_3),(v_4,v_4),(v_5,v_5),(v_1,-v_2),(v_2,v_1+v_2),(v_3,-v_1-v_2-v_3-v_5),(v_4,v_5),(v_5,-v_4)\}$.  These vectors give the Gram matrix
\begin{equation*}
G_{K'}=
\begin{pmatrix}
4&-2&0&0&0&3&0&-1&0&0\\
-2&4&-2&0&0&-3&3&-1&0&0\\
0&-2&4&-2&-2&1&-2&2&-2&0\\
0&0&-2&4&0&0&0&0&2&-2\\
0&0&-2&0&4&0&0&-2&2&2\\
3&-3&1&0&0&4&-2&0&0&0\\
0&3&-2&0&0&-2&4&-2&0&0\\
-1&-1&2&0&-2&0&-2&4&-2&-2\\
0&0&-2&2&2&0&0&-2&4&0\\
0&0&0&-2&2&0&0&-2&0&4
\end{pmatrix},
\end{equation*}
which has determinant 256 and Smith invariants 1, 1, 1, 1, 1, 1, 2, 2, 8, 8.

\subsection{3-Cycle with Negation of a Disjoint Coordinate} \label{3cycleneg}
Let $h$ act as the 3-cycle $(e_1,e_2,e_3)$, negate $e_4$, and fix $e_5$.  Then the lattice $K$ obtained from $\mathbb{Z}^5$ is isometric to $A_5\perp {A_1}^3$.  The sublattice $K'$ is spanned by $\{(v_1,v_1),(v_2,v_2),(v_3,v_3),(v_4,v_4),\\(v_5,v_5),(v_1,v_2),(v_2,-v_1-v_2),(v_3,v_1+v_2+v_3+v_4+v_5),(v_4,-v_5),(v_5,-v_4)\}$.  We may write $(v_1,v_2)=(v_1,v_1)+2(v_2,v_2)+3(v_3,v_3)+3(v_4,v_4)-2(v_2,-v_1-v_2)-3(v_3,v_1+v_2+v_3+v_4+v_5)-3(v_4,-v_5)$ and $(v_5,-v_4)=-(v_4,v_4)+(v_5,v_5)+(v_4,-v_5)$, and the remaining vectors give the Gram matrix
\begin{equation*}
G_{K'}=
\begin{pmatrix}
4&-2&0&0&0&-2&1&0\\ 
-2&4&-2&0&0&1&-1&0\\
0&-2&4&-2&-2&0&1&0\\
0&0&-2&4&0&0&0&2\\
0&0&-2&0&4&0&0&-2\\
-2&1&0&0&0&4&-2&0\\ 
1&-1&1&0&0&-2&4&-2\\
0&0&0&2&-2&0&-2&4
\end{pmatrix},
\end{equation*}
which has determinant 192 and Smith invariants 1, 1, 1, 1, 2, 2, 4, 12.

\subsection{3-Cycle with Negation of Two Disjoint Coordinates}
Let $h$ act as the 3-cycle $(e_1,e_2,e_3)$ and negate $e_4$ and $e_5$.  Then the lattice $K$ obtained from $\mathbb{Z}^5$ is isometric to $A_5\perp {A_1}^4$.  The sublattice $K'$ is spanned by $\{(v_1,v_1),(v_2,v_2),(v_3,v_3),(v_4,v_4),\\(v_5,v_5),(v_1,v_2),(v_2,-v_1-v_2),(v_3,v_1+v_2+v_3+v_4+v_5),(v_4,-v_4),(v_5,-v_5)\}$.  We may write
\begin{align}
(v_1,0)=&(v_1,v_1)+(v_2,v_2)+2(v_3,v_3)+(v_4,v_4)+(v_5,v_5)-(v_2,-v_1-v_2)\nonumber\\&{}-2(v_3,v_1+v_2+v_3+v_4+v_5)-(v_4,-v_4)-(v_5,-v_5)\nonumber\\
(0,v_1)=&(v_1,v_1)-(v_1,0)\nonumber\\
(0,v_2)=&(v_1,v_2)-(v_1,0)\nonumber\\
(v_2,0)=&(v_2,v_2)-(0,v_2)\nonumber\\
(0,v_4+v_5)=&(v_3,v_1+v_2+v_3+v_4+v_5)-(v_3,v_3)-(0,v_1)=(0,v_2)\nonumber\\
(v_5,-v_5)=&(v_4,v_4)+(v_5,v_5)-(v_4,-v_4)-2(0,v_4+v_5),\nonumber
\end{align}
so $K$ is spanned by $\{(v_1,0),(0,v_1),(v_2,0),(0,v_2),(v_3,v_3),(v_4,v_4),(v_4,-v_4),(v_5,v_5),(0,v_4+v_5)\}$.  These vectors give the Gram matrix
\begin{equation*}
G_{K'}=
\begin{pmatrix}
2&0&-1&0&0&0&0&0&0\\
0&2&0&-1&0&0&0&0&0\\
-1&0&2&0&-1&0&0&0&0\\
0&-1&0&2&-1&0&0&0&0\\
0&0&-1&-1&4&-2&0&-2&-2\\
0&0&0&0&-2&4&0&0&2\\
0&0&0&0&0&0&4&0&-2\\
0&0&0&0&-2&0&0&4&2\\
0&0&0&0&-2&2&-2&2&4
\end{pmatrix},
\end{equation*}
which has determinant 384 and Smith invariants 1, 1, 1, 1, 2, 2, 2, 2, 24.

\subsection{Negative of 3-Cycle, Fixing the Remaining Coordinates}
Take $h$ to be the negative of the above element.  Then $K\cong D_6 \perp {A_1}^2$.  The sublattice $K'$ is spanned by $\{(v_1,v_1),(v_2,v_2),(v_3,v_3),(v_4,v_4),(v_5,v_5),(v_1,-v_2),(v_2,v_1+v_2),(v_3,-v_1-v_2-v_3-v_4-v_5)\}$.  These vectors give the Gram matrix
\begin{equation*}
G_{K'}=
\begin{pmatrix}
4&-2&0&0&0&3&0&-1\\
-2&4&-2&0&0&-3&3&-1\\
0&-2&4&-2&-2&1&-2&3\\
0&0&-2&4&0&0&0&-2\\
0&0&-2&0&4&0&0&-2\\
3&-3&1&0&0&4&-2&0\\
0&3&-2&0&0&-2&4&-2\\
-1&-1&3&-2&-2&0&-2&4
\end{pmatrix},
\end{equation*}
which has determinant 64 and Smith invariants 1, 1, 1, 1, 2, 2, 4, 4.

\subsection{Negative of 3-Cycle with Negation of One Disjoint Coordinate}
Take $h$ to be the negative of the element in section \ref{3cycleneg}.  Then the lattice $K$ obtained from $\mathbb{Z}^5$ is isometric to $D_6 \perp {A_1}^3$.  The sublattice $K'$ is spanned by $\{(v_1,v_1),(v_2,v_2),(v_3,v_3),(v_4,v_4),\\(v_5,v_5),(v_1,-v_2),(v_2,v_1+v_2),(v_3,-v_1-v_2-v_3-v_4-v_5),(v_4,v_5),(v_5,v_4)\}$.  We have $(v_5,v_4)=(v_4,v_4)+(v_5,v_5)-(v_4,v_5)$,  and the remaining vectors give the Gram matrix
\begin{equation*}
G_{K'}=
\begin{pmatrix}
4&-2&0&0&0&3&0&-1&0\\
-2&4&-2&0&0&-3&3&-1&0\\
0&-2&4&-2&-2&1&-2&3&-2\\
0&0&-2&4&0&0&0&-2&2\\
0&0&-2&0&4&0&0&-2&2\\
3&-3&1&0&0&4&-2&0&0\\
0&3&-2&0&0&-2&4&-2&0\\
-1&-1&3&-2&-2&0&-2&4&-2\\
0&0&-2&2&2&0&0&-2&4
\end{pmatrix},
\end{equation*}
which has determinant 128 and Smith invariants 1, 1, 1, 1, 2, 2, 2, 2, 8.

\subsection{4-Cycle with a Negation}
Let $h$ be the element which sends $e_1$ to $e_2$, $e_2$ to $e_3$, $e_3$ to $e_4$, $e_4$ to $-e_1$, and fixes $e_5$.  Considering only the first four basis vectors, we obtain the Gram matrix
\begin{equation*}
\begin{pmatrix}
2&0&0&0&1&0&0&-1\\
0&2&0&0&1&1&0&0\\
0&0&2&0&0&1&1&0\\
0&0&0&2&0&0&1&1\\
1&1&0&0&2&0&0&0\\
0&1&1&0&0&2&0&0\\
0&0&1&1&0&0&2&0\\
-1&0&0&1&0&0&0&2
\end{pmatrix},
\end{equation*}
which has determinant 4.  Since this lattice is even and contained in $\mathbb{Z}^8$, it is isometric to $D_8$, so we have $K\cong D_8 \perp A_1$.  The sublattice $K'$ is spanned by $\{(v_1,v_1),(v_2,v_2),(v_3,v_3),(v_4,v_4),(v_5,v_5),(v_1,v_2),(v_2,v_3),(v_3,v_1+v_2+v_3+v_4+v_5),(v_4,-v_1-v_2-v_3-v_5),(v_5,-v_1-v_2-v_3-v_4)\}$.  We may write $(v_5,-v_1-v_2-v_3-v_4)=-(v_4,v_4)+(v_5,v_5)+(v_4,-v_1-v_2-v_3-v_5)$, and the remaining vectors are linearly independent and give the Gram matrix
\begin{equation*}
G_{K'}=
\begin{pmatrix}
4&-2&0&0&0&1&-1&1&-1\\
-2&4&-2&0&0&1&1&-1&0\\
0&-2&4&-2&-2&-1&1&1&-1\\
0&0&-2&4&0&0&-1&0&3\\
0&0&-2&0&4&0&-1&0&-1\\
1&1&-1&0&0&4&-2&0&0\\
-1&1&1&-1&-1&-2&4&-2&0\\
1&-1&1&0&0&0&-2&4&-2\\
-1&0&-1&3&-1&0&0&-2&4
\end{pmatrix},
\end{equation*}
which has determinant 32 and Smith invariants 1, 1, 1, 1, 1, 1, 2, 2, 8.

\subsection{Negative of the Above}
Let $h$ be the negative of the above element.  Then $K\cong D_8 \perp {A_1}^2$, and the sublattice $K'$ is spanned by $\{(v_1,v_1),(v_2,v_2),(v_3,v_3),(v_4,v_4),(v_5,v_5),(v_1,-v_2),(v_2,-v_3),(v_3,-v_1-v_2-v_3-v_4-v_5),(v_4,v_1+v_2+v_3+v_5),(v_5,v_1+v_2+v_3+v_4)\}$.  These vectors give the Gram matrix
\begin{equation*}
G_{K'}=
\begin{pmatrix}
4&-2&0&0&0&1&-1&1&-1\\
-2&4&-2&0&0&1&1&-1&0\\
0&-2&4&-2&-2&-1&1&1&-1\\
0&0&-2&4&0&0&-1&0&3\\
0&0&-2&0&4&0&-1&0&-1\\
1&1&-1&0&0&4&-2&0&0\\
-1&1&1&-1&-1&-2&4&-2&0\\
1&-1&1&0&0&0&-2&4&-2\\
-1&0&-1&3&-1&0&0&-2&4
\end{pmatrix},
\end{equation*}
which has determinant 256 and Smith invariants 1, 1, 1, 1, 1, 1, 2, 2, 8, 8.

\section{Conclusion}
The cases given above are only a small sampling of the potential of this construction.  These cases, along with the general results given in section \ref{General}, are the results of a short investigation of the procedure applied to $\mathbb{Z}^n$ for small $n$, and to root lattices of small rank.  Continuation of this investigation may lead to more significant results on higher-rank lattices as well as more generalizations.  Also, many lattices arose which were not identified.  Further research may find better ways to describe these lattices.

\end{document}